\documentclass[11pt,leqno]{amsart}
\usepackage{amssymb,amsmath,amsthm,amsfonts}
\usepackage{graphicx}

\DeclareMathOperator{\Real}{Re} 
\DeclareMathOperator{\spa}{span} 
 \DeclareMathOperator{\diver}{div}

\DeclareMathOperator{\Id}{\mathbf{Id}}
\DeclareMathOperator{\End}{End}\DeclareMathOperator{\trace}{trace}
\DeclareMathOperator{\GL}{GL}
\DeclareMathOperator{\Adm}{Adm}

\newcommand{\field}[1]{\mathbb{#1}}
\newcommand{\G}{\field{G}}                      
\newcommand{\F}{\field{F}}                      
\newcommand{\R}{\field{R}}                      
\newcommand{\C}{\field{C}}                      
\newcommand{\Heis}{\field{H}}                    

\newcommand{\eps}{\epsilon}

\newcommand{\ff}{{\mathfrak f}}

\newcommand{\fg}{{\mathfrak g}}

\newcommand{\fh}{{\mathfrak h}}
\newcommand{\heis}{{\mathfrak h}}

\newcommand{\cL}{{\mathcal L}}

\newcommand{\fv}{{\mathfrak v}}

\newcommand{\tta}{{\tt a}}
\newcommand{\ttg}{{\tt g}}
\newcommand{\tth}{{\tt h}}
\newcommand{\cX}{{\mathcal X}}

\newcommand{\boldA}{{\mathbf A}}

\newcommand{\boldB}{{\mathbf B}}

\newcommand{\boldC}{{\mathbf C}}

\newcommand{\bolde}{{\mathbf e}}
\newcommand{\boldJ}{{\mathbf J}}

\newcommand{\boldN}{{\mathbf N}}
\newcommand{\boldO}{{\mathbf O}}

\newcommand{\boldS}{{\mathbf S}}
\newcommand{\boldt}{{\mathbf t}}

\newcommand{\boldx}{{\mathbf x}}
\newcommand{\boldX}{{\mathbf X}}
\newcommand{\bx}{{\mathbf x}}
\newcommand{\bt}{{\mathbf t}}

\newcommand{\bb}{{\mathbf b}}
\newcommand{\boldOmega}{{\mathbf \Omega}}

\newcommand{\norm}[1]{\lVert #1 \rVert}

\newcommand{\deriv}[1]{{\frac{\partial}{\partial #1}}}

\newcommand{\bi}{{\mathbf i}}

\def\Barint_#1{\mathchoice
          {\mathop{\vrule width 6pt height 3 pt depth -2.5pt
                  \kern -8pt \intop}\nolimits_{#1}}%
          {\mathop{\vrule width 5pt height 3 pt depth -2.6pt
                  \kern -6pt \intop}\nolimits_{#1}}%
          {\mathop{\vrule width 5pt height 3 pt depth -2.6pt
                  \kern -6pt \intop}\nolimits_{#1}}%
          {\mathop{\vrule width 5pt height 3 pt depth -2.6pt
                  \kern -6pt \intop}\nolimits_{#1}}}

\theoremstyle{plain}
\newtheorem{theorem}{Theorem}
\newtheorem{corollary}[theorem]{Corollary}
\newtheorem{lemma}[theorem]{Lemma}
\newtheorem{proposition}[theorem]{Proposition}

\theoremstyle{definition}
\newtheorem{definition}[theorem]{Definition}
\newtheorem{example}[theorem]{Example}
\newtheorem{remark}[theorem]{Remark}

\newtheorem{conjecture}[theorem]{Conjecture}

\numberwithin{theorem}{section} \numberwithin{equation}{section}

\title{Stability theorems for H-type Carnot groups}
\date{\today}
\author{Jeremy T. Tyson}
\address{Department of Mathematics \\ University of Illinois at Urbana-Champaign \\ 1409 West Green Street \\ Urbana, IL 61801}
\email{tyson@illinois.edu}

\begin{document}
\maketitle

\begin{abstract}
We introduce the {\it H-type deviation} of a step two Carnot group $\G$. This quantity, denoted $\delta(\G)$, measures the deviation of $\G$ from the class of H-type (Heisenberg-type) groups. More precisely, we show that $\delta(\G)=0$ if and only if $\G$ carries a vertical metric which endows it with the structure of an H-type group. We compute the H-type deviation for several naturally occurring families of step two Carnot groups. In addition, we provide several analytic expressions which are comparable to the H-type deviation of $\G$. As a consequence, we establish new analytic characterizations for the class of H-type groups. For instance, denoting by $N(\ttg) = (||\boldx||_h^4 + 16 ||\boldt||_v^2)^{1/4}$, $\ttg = \exp(\boldx+\boldt)$, the canonical Kaplan-type quasi-norm in a step two group $\G$ with taming Riemannian metric $g = g_h \oplus g_v$, we show that $\G$ is H-type if and only if $||\nabla_0 N(\ttg)||_h^2 = ||\boldx||_h^2/N(\ttg)^2$ for all $\ttg \ne 0$. Similarly, we show that $\G$ is H-type if and only if $N^{2-Q}$ is $\cL$-harmonic in $\G \setminus \{0\}$. Here $\nabla_0$ denotes the horizontal differential operator, $\cL$ denotes the canonical sub-Laplacian, and $Q = \dim\fv_1+2\dim\fv_2$ is the homogeneous dimension of $\G$, where $\fv_1\oplus\fv_2$ is the stratification of the Lie algebra of $\G$. It is well-known that H-type groups satisfy both of these analytic conclusions. The new content of these results lies in the converse directions. Motivation for this work comes from a longstanding conjecture regarding polarizable Carnot groups. We formulate a quantitative stability conjecture regarding the fundamental solution for the sub-Laplacian on step two Carnot groups. Its validity would imply that all step two polarizable groups admit an H-type group structure. We confirm this conjecture for a sequence of anisotropic Heisenberg groups.
\end{abstract}

\tableofcontents

\section{Introduction}

The analytic theory of stratified Lie groups has been a subject of intense study ever since the foundational work of Folland, Rothschild, and Stein in the 1970s. When equipped with a sub-Riemannian structure, as determined by a metric on the horizontal (first) layer of the stratified Lie algebra, such groups serve as fertile testing grounds for extensions of first-order differential analysis and geometry to metric measure spaces. Standard references in this area include the books of Folland and Stein \cite{fs:hardy} and Bonfiglioli--Lanconelli--Uguzzoni \cite{blu}, as well as the later chapters in Stein's well-known harmonic analysis reference \cite{Ste}.

The Heisenberg group $\Heis^n$ was the progenitor for these developments, and continues to serve as a primary exemplar within the field. The existence, in this setting, of a large symmetry group makes the underlying analysis and geometry particularly elegant, and allows for simplified and streamlined arguments, whose analogs in the case of general stratified Lie groups are often substantially more involved. Folland's identification in \cite{fol:explicit} of an explicit fundamental solution for the sub-Laplacian in $\Heis^n$ was a highly influential development. Denoting by
\begin{equation}\label{eq:XY}
X_1,\ldots,X_n,Y_1,\ldots,Y_n
\end{equation}
an orthonormal basis for the first layer $\fv_1$ of the stratified Lie algebra $\fh_n = \fv_1 \oplus \fv_2$ of $\Heis^n$, the sub-Laplacian is the second order subelliptic operator
\begin{equation}\label{eq:L}
\cL = \sum_{i=1}^n X_i^2 + Y_i^2.
\end{equation}
This operator plays a role in analysis on the Heisenberg group of similar importance to that played by the usual Laplacian in Euclidean space. Let us denote by $T$ a basis element for the second layer $\fv_2$ of the Lie algebra, and let us recall that the nontrivial bracket relations among the elements $X_i,\ldots,X_n,Y_1,\ldots,Y_n,T$ are $[X_i,Y_i]=T$. We introduce coordinates $(x_1,\ldots,x_n,y_1,\ldots,y_n,t)$ in $\Heis^n$ by identifying the group with its Lie algebra via the exponential map and reading coordinates in the aforementioned basis. Denote by $z = (x_1,\ldots,x_n,y_1,\ldots,y_n)$ the point in $\R^{2n}$ corresponding to the horizontal basis vectors \eqref{eq:XY}. With this notation in place, Folland's result reads as follows.

\begin{theorem}[Folland]\label{th:folland}
For each $n\ge 1$, there exists a constant $c(n)$ so that
\begin{equation}\label{eq:folland-solution}
u(z,t) = c(n) (||z||^4 + 16 t^2)^{-n/2}
\end{equation}
is a fundamental solution for the sub-Laplacian $\cL$ on $\Heis^n$.
\end{theorem}

Sub-Riemannian analysis in the Heisenberg groups is intimately related to large-scale and asymptotic geometry in complex hyperbolic space. This story extends to arbitrary rank one symmetric spaces of noncompact type, which encompasses hyperbolic geometry over the real, complex and quaternionic numbers as well as (in a single exceptional case) over the octonions. With the exception of real hyperbolic space, the asymptotic geometry of these spaces is modeled on sub-Riemannian analysis in a collection of step two sub-Riemannian stratified Lie groups, the {Iwasawa groups}. The class of Iwasawa groups, which includes the Heisenberg groups, maintains many nice features of sub-Riemannian Heisenberg analysis. Strictly larger still is the class of Heisenberg-type groups (for short, H-type groups), introduced by Kaplan \cite{kap:h-type}. The H-type condition involves certain algebraic identities on the level of the Lie algebra, formulated in terms of the {J operator}. Recall that the J operator acts on elements of the second layer $\fv_2$ of the Lie algebra and returns endomorphisms of the first layer $\fv_1$. This operator is uniquely characterized by the identity
$$
\langle \boldJ_T(U),V \rangle_h = \langle T,[U,V] \rangle_v \qquad \forall \, U,V \in \fv_1, \, T \in \fv_2 \, .
$$
Here we have fixed an inner product $g = g_h \oplus g_v$ on $\fg$ which makes the two layers $\fv_1$ and $\fv_2$ orthogonal, equivalently, a left-invariant taming Riemannian metric adapted to the underlying sub-Riemannian structure. We emphasize that the J operator depends on the choice of this taming Riemannian metric, and we write $\boldJ = \boldJ^{[g_v]}$ when necessary to stress this point.

\begin{definition}[Kaplan]
A step two stratified Lie group $\G$ equipped with a taming Riemannian metric $g = g_h \oplus g_v$ is of {\bf Heisenberg type} (or {\bf H-type}) if $\boldJ_T$ is an orthogonal transformation of $(\fv_1,g_h)$ whenever $T \in \fv_2$ satisfies $||T||_v = 1$.
\end{definition}

In view of skew-symmetry and the linearity properties of the J operator, it is easy to see that $\G$ is an H-type group if and only if
\begin{equation}\label{eq:Htypecondition}
\boldJ_T^2 = - ||T||_v^2 \, \Id
\end{equation}
for all $T \in \fv_2$, where $\Id$ denotes the identity transformation acting on $\fv_1$.

A important distinguishing feature of H-type groups is the existence therein of an explicit formula for the fundamental solution of the canonical sub-Laplacian, analogous to \eqref{eq:folland-solution}. Indeed, denoting by $X_1,\ldots,X_m$ a $g_h$-orthonormal basis for the first layer $\fv_1$ and by $T_1,\ldots,T_{m_2}$ a $g_v$-orthonormal basis for the second layer $\fv_2$, and introducing exponential coordinates $(\boldx,\boldt) = (x_1,\ldots,x_m,t_1,\ldots,t_{m_2})$ in $\G$, an exact analog of Theorem \ref{th:folland} holds. 

\begin{theorem}[Kaplan, \cite{kap:h-type}]\label{th:kaplan-theorem}
In any H-type group $\G$, the function
\begin{equation}\label{eq:kaplan-theorem}
u(\boldx,\boldt) = (||\boldx||_h^4 + 16 ||\boldt||_v^2)^{(2-Q)/4}
\end{equation}
is (up to a constant multiple) a fundamental solution for the sub-Laplacian
$$
\cL = X_1^2 + \cdots + X_m^2
$$
\end{theorem}

Here $Q = m + 2m_2$ denotes the homogeneous dimension of $\G$. 

The proof of Theorem \ref{th:kaplan-theorem} relies on a suite of explicit algebraic identities which hold in H-type groups. These algebraic identities lie at the heart of many other explicit formulas in, and distinguishing features of, H-type groups. For instance, Capogna--Danielli--Garofalo \cite{cdg:carnot} showed that, for each $1<p\le\infty$, a suitable power of the quantity
$$
N(\boldx,\boldt) = (||\boldx||_h^4 + 16 ||\boldt||_v^2)^{1/4}
$$
is, up to a constant multiple, the fundamental solution $u_p$ for the nonlinear $p$-Laplacian $\cL_p$. (See also Heinonen--Holopainen \cite{hh:carnot} for the case $p=Q$, where the power is replaced by a logarithm.) A key intermediate step in this derivation is to show that the horizontal gradient of $N$ satisfies
\begin{equation}\label{eq:norm-horiz-grad}
||\nabla_0 N(\boldx,\boldt)||_h^2 = \frac{||\boldx||_h^2}{N_v(\boldx,\boldt)^2}.
\end{equation}
This identity is an analog of the {\it eikonal equation} $|\nabla n| \equiv 1$ for the usual Euclidean norm $n(\boldx) = |\boldx|$ in $\R^n$. In \cite{gv:yamabe} it is shown that in each H-type group $\G$ and for each $\epsilon>0$, a suitable multiple of the function
$$
v(\boldx,\boldt) = ((\epsilon^2 + ||\boldx||_h^2)^2 + 16 ||\boldt||_v^2)^{(2-Q)/4}
$$
is a positive solution to the Yamabe equation
$$
-\cL v = v^{(Q+2)/(Q-2)}
$$
defined in all of $\G$. Finally, in \cite{gt:h-type} it is shown in all H-type groups that nonnegative finite linear combinations of translates of the aforementioned $p$-Laplacian fundamental solutions $u_p$ satisfy a family of monotonicity conclusions, e.g., are either $p$-subsolutions or $p$-supersolutions depending on the value of $p$ relative to the homogeneous dimension $Q$. This fact extends a remarkable superposition formula in Euclidean space due to Crandall--Zhang \cite{cz:another}; see also Lindqvist--Manfredi \cite{lm:remarkable} and Brustad \cite{bru:superposition1}, \cite{bru:superposition2}.

Are there any other sub-Riemannian Carnot groups (of step at least two) besides the H-type groups where a similar array of explicit formulas hold true? It is a well known result of Folland \cite{fol:subelliptic} that in any stratified Lie group $\G$ (of any step) there always exists a fundamental solution $u$ for the sub-Laplacian $\cL = X_1^2 + \cdots + X_m^2$, moreover, $u$ is homogeneous of order $2-Q$ with respect to the standard dilations of $\G$. Here, as before, $Q$ denotes the homogeneous dimension of $\G$.
The paper \cite{bt:polar} introduced the class of {\bf polarizable Carnot groups}. These groups were defined by the assumption that the quasinorm 
$$
N = u^{1/(2-Q)},
$$
associated to Folland's fundamental solution $u$ of the sub-Laplacian $\cL$, is annihilated by the infinity-Laplacian $\cL_\infty$ away from the identity element $0$. It easily follows from this assumption that a full one-parameter family of functions $u_p$, obtained as suitable powers or as the logarithm of the quasinorm $N$, continue to serve as fundamental solutions for the $p$-Laplacians. 

The term {\it polarizable} originates from a further geometric property of such groups. Namely, any polarizable Carnot group $\G$ admits a family of half-infinite horizontal curves $\gamma_\tta:[0,\infty) \to \G$ with $\gamma_\tta(0) = 0$ and $\gamma_\tta(1) = \tta$, which fill out a subset of $\G$ of full Haar measure, so that the following {\it horizontal polar coordinate integration formula} holds true:
\begin{equation}\label{eq:polar}
\int_{\G} f(\ttg) \, d\ttg = \int_0^\infty \int_S f(\gamma_\tta(s)) s^{Q-1} \, d\sigma(\tta) \, ds \qquad \mbox{for all $f \in L^1(\G)$.}
\end{equation}
Here $S = \{ \tta \in \G \, : \, N(\tta) = 1\}$ denotes the $N$-unit sphere and $\sigma$ denotes a specific Radon measure supported on $S$. Other polar coordinate integration formulas are known to hold in stratified Lie groups (e.g.\ \cite[Proposition 1.15]{fs:hardy}), but in contrast with those formulas \eqref{eq:polar} uses {\it horizontal} radial curves $\gamma_\tta$. This feature in turn enables the explicit computation of moduli of ring domains and other notions in geometric mapping theory. Recently (\cite{tys:polar2}), we have shown that the validity of such a horizontal polar coordinate integration formula is not just a consequence of the polarizability condition, but is in fact equivalent to it. Thus the validity of \eqref{eq:polar} for a suitable family of horizontal curves and a suitable Radon measure on the parametrizing sphere $N$ can be taken as a definition for polarizability of a Carnot group. We note that Kor\'anyi and Reimann \cite{kr:rings} were the first to construct a horizontal polar coordinate integration formula in the Heisenberg group $\Heis^1$.

In this paper, we return to the question of whether these well-known explicit analytic formulas in H-type groups extend to other settings. We are also motivated by a question posed in \cite{bt:polar}, namely, whether all polarizable Carnot groups are of Heisenberg type. We introduce here a strategy which may be useful to answer this question in the setting of step two Carnot groups. To wit, we introduce a numerical measurement which quantifies the degree to which a given step two Carnot group $\G$ deviates from the subclass of H-type groups. More precisely (since a step two Carnot group is defined only in terms of a metric $g_h$ on the horizontal layer $\fv_1$ of the Lie algebra, while the H-type condition requires a full Riemannian metric $g_h \oplus g_v$ on the stratified Lie algebra $\fg = \fv_1 \oplus \fv_2$), our numerical measurement quantifies the deviation of $\G$ from the class of {\it nascent H-type groups}. A step two Carnot group $\G$, with horizontal metric $g_h$ defined in the first layer $\fv_1$, is said to be a {\bf nascent H-type group} if there exists a vertical metric $g_v$ defined in the second layer $\fv_2$ so that $\G$, equipped with the Riemannian metric $g_h \oplus g_v$, is an H-type group.

The {\bf H-type deviation} of $\G$ is defined to be
$$
\delta(\G) = \frac1{\sqrt{\dim\fv_1}} \, \inf_{g_v} \sup_{\substack{T \in \fv_2 \\ ||T||_v = 1}} ||\boldJ_T^2 + \Id||_{HS}.
$$
Here $||\cdot||_{HS}$ denotes the Hilbert--Schmidt matrix norm, and the infimum is taken over all vertical metrics $g_v$ defined in $\fv_2$. The factor $(\dim\fv_1)^{-1/2}$ is a normalization factor which takes into account the overall size of the group. 
We shall show (see Proposition \ref{prop:delta_0-characterization-of-nascent-H-type-groups}) that $\delta(\G) = 0$ if and only if $\G$ is a nascent H-type group.

We also compute the H-type deviation for several naturally occuring families of step two Carnot groups. For instance, we compute this value for all anisotropic Heisenberg groups (and more generally, for all generalized Heisenberg groups in the sense of \cite{br:generalized-H-type}) and we also compute its value for step two free Carnot groups of arbitrary rank. See subsection \ref{sec:examples} for details.

The main results of this paper point towards the use of this concept to resolve the polarizability conjecture. We show that some previously mentioned analytic properties of Kaplan's quasinorm 
\begin{equation}\label{eq:Kaplan-N}
N_v(\boldx,\boldt) = (||\boldx||_h^4 + 16 ||\boldt||_v^2)^{1/4}
\end{equation}
characterize H-type groups. More precisely, we prove the following two theorems, cf.\ \eqref{eq:norm-horiz-grad} and Theorem \ref{th:folland}.

\begin{theorem}\label{thm:main-theorem-1}
Let $\G$ be a step two sub-Riemannian Carnot group, with metric $g_h$ defined in the horizontal layer $\fv_1$. Assume that there exists a metric $g_v$ defined in the vertical layer $\fv_2$ so that
$$
||(\nabla_0 N_v)(\ttg)||_h^2 = \frac{||\boldx||_h^2}{N_v(\ttg)^2} \quad \forall \ttg = \exp(\boldx + \boldt) \in \G, \ttg \ne 0 \, .
$$
where $N_v$ is as in \eqref{eq:Kaplan-N}. Then $\G$ is an H-type group.
\end{theorem}

\begin{theorem}\label{thm:main-theorem-2}
Let $\G$ be as in Theorem \ref{thm:main-theorem-1}. Assume there is a metric $g_v$ defined in $\fv_2$ so that
$$
\cL(N_v^{2-Q})(\ttg) = 0 \quad \forall 0 \ne \ttg \in \G,
$$
where $N_v$ is as in \eqref{eq:Kaplan-N}. Then $\G$ is an H-type group.
\end{theorem}

Theorem \ref{thm:main-theorem-1} follows from the aforementioned characterization of nascent H-type groups via the vanishing of the H-type deviation $\delta_0(\G)$, along with the quantitative stability estimate
$$
\frac1C \delta(\G) \le \inf_{g_v} \sup_{0 \ne \ttg \in \G} \Biggl| ||(\nabla_0 N_v)(\ttg)||_h^2 - \frac{||\boldx(\ttg)||_h^2}{N_v(\ttg)^2} \Biggr| \le C \delta(\G)
$$
where the infimum is taken over all vertical metrics $g_v$ and $C>0$ denotes a constant which depends only on the dimensions of the horizontal and vertical layers of $\G$. Similarly, Theorem \ref{thm:main-theorem-2} relies on an analogous quantitative estimate of the form
\begin{equation}\label{eq:main-theorem-2-equation}
\frac1C \delta(\G) \le \inf_{g_v} \sup_{0 \ne \ttg \in \G} N_v(\ttg)^Q (\cL N_v^{2-Q})(\ttg) \le C \delta(\G).
\end{equation}
The quantitative estimate in \eqref{eq:main-theorem-2-equation} provides new insight into Kaplan's theorem. Recall that the sub-Laplacian $\cL$ in a Carnot group $\G$ depends only on horizontal data and the choice of volume measure (Haar measure), but that Kaplan's formula \eqref{eq:kaplan-theorem} for the fundamental solution involves a specific choice of vertical metric. While the H-type condition depends essentially on that choice of vertical metric, it is natural to wonder how such choice arises in the context of the sub-Laplacian. The estimates in \eqref{eq:main-theorem-2-equation} provide a partial answer. Indeed, the characterization of the H-type condition via the vanishing of the H-type deviation, coupled with the estimates in \eqref{eq:main-theorem-2-equation}, illustrate the role of this specific choice of vertical metric in the fundamental solution for $\cL$.

The technique used to prove Theorems \ref{thm:main-theorem-1} and \ref{thm:main-theorem-2} could be employed to prove a variety of similar statements. We have not systematically explored this topic in this paper. However, we conclude this introduction by indicating how this concept, and the techniques used in the proofs of Theorems \ref{thm:main-theorem-1} and \ref{thm:main-theorem-2}, could be used to attack the polarizability conjecture.

In \cite{bgg:step2}, Beals, Gaveau and Greiner provide a semi-explicit representation for the fundamental solution of the sub-Laplacian $\cL$ in any step two Carnot group $\G$. More precisely, they show that the fundamental solution $u(\boldx,\boldt)$ for $\cL$ admits the integral representation
\begin{equation}\label{eq:bgg-formula}
u(\boldx,\boldt) = c(Q) \int_{\fv_2} \frac{V(\tau)}{f(\boldx,\boldt,\tau)^{(Q/2)-1}} \, d\tau
\end{equation}
where the constant $c(Q)>0$ depends only on the homogeneous dimension $Q = m + 2m_2$ of $\G$. Here $\tau$ denotes a parameterizing vector in $\fv_2$, $V(\tau)$ denotes the solution to the {\it generalized transport equation} $\tau \cdot \nabla_\tau V + \tfrac12 (\cL f - m) V = 0$, and $f(\boldx,\boldt,\tau)$ denotes the solution to the {\it generalized Hamilton--Jacobi equation} $\tfrac12 ||\nabla_0 f||_h^2 + \tau \cdot \nabla_\tau f = f$. Explicit expressions for $V(\tau)$ and $f(\boldx,\boldt,\tau)$ involve the matrix-valued parameter 
$$
\boldOmega(\tau) := -\tfrac12 \sqrt{-1} \, \sum_{q=1}^{m_2} \boldB^q \cdot \tau_q,
$$
where the vector $(\tau_q)$ records the coordinates of $\tau$ relative to a suitable basis of $\fv_2$ and $\boldB^q$, $1\le q \le m_2$, are the structure constant matrices for the Lie algebra $\fg$ of $\G$. See \cite{bgg:step2} or \cite{bt:polar} for more details.

In \cite{bt:polar} we explicitly computed the integral in \eqref{eq:bgg-formula} in the case of the anisotropic Heisenberg group $\G = \Heis^2(\tfrac12,1)$. (See Example \ref{ex:heisenberg-example} for the definition.) We then showed that
$$
\cL_\infty \bigl( u^{\tfrac1{2-Q}} \bigr)
$$
does not identically vanish in $\G$, and concluded that the anisotropic Heisenberg group $\Heis^2(\tfrac12,1)$ is not a (nascent) H-type group. Bou Dagher and Zegarli\'nski \cite{bdz:anisotropic} gave an explicit formula for Folland's fundamental solution in the anisotropic Heisenberg groups $\Heis^n(\tfrac12,1,\ldots,1)$. In section \ref{sec:polarizability} we analyze the $\infty$-Laplacian of the Bou Dagher--Zegarli\'nski fundamental solution on $\Heis^n(\tfrac12,1,\ldots,1)$, and compare its behavior to the H-type deviation of this group. Based on these conclusions, we formulate the following conjecture.

\begin{conjecture}\label{conj:polarizability-conjecture}
Let $\G$ be a step two Carnot group with horizontal metric $g_h$. Denote by $\Sigma = \{\ttg = \exp(\boldx + \boldt) \in \G : \boldt = 0 \}$. Let $u$ be the fundamental solution for the sub-Laplacian $\cL$ and let $N = u^{1/(2-Q)}$. Then
$$
\sup_{\ttg \in \Sigma} |N(\ttg)(\cL_\infty N)(\ttg)| \simeq \delta(\G)^2 \, .
$$
\end{conjecture}

Conjecture \ref{conj:polarizability-conjecture}, if true, would imply that all step two polarizable Carnot groups are H-type groups. In Theorem \ref{th:anisotropic-Heis-groups} we verify Conjecture \ref{conj:polarizability-conjecture} for the groups $\Heis^n(\tfrac12,1,\ldots,1)$.

\smallskip

\noindent \paragraph{\bf Acknowledgements.} The author acknowledges support from the Simons Foundation under grant \#852888, `Geometric mapping theory and geometric measure theory in sub-Riemannian and metric spaces'. In addition, this material is based upon work supported by and while the author was serving as a Program Director at the National Science Foundation. Any opinion, findings, and conclusions or recommendations expressed in this material are those of the author and do not necessarily reflect the views of the National Science Foundation.

\section{Background and preliminary definitions}\label{sec:background}

\subsection{Notation}

Vectors, operators and matrices will be denoted in boldface. Real-valued functions will be denoted with regular lower case Greek letters. The expression $|\cdot|$ denotes the absolute value function in $\R$, while $|\cdot|_2$ denotes the Euclidean norm in $\R^n$ for any $n$. 

The expression $||\cdot||$ (decorated with subscripts) denotes various norms on operators or matrices. For example, for an $m\times m$ matrix $\boldA$ we denote by $||\boldA||_{op}$ the operator norm of $\boldA$ and by $||\boldA||_{HS} = \trace(\boldA^t\boldA)^{1/2}$ the Hilbert--Schmidt norm of $\boldA$. We recall that
\begin{equation}\label{eq:op-HS}
||\boldA||_{op} \le ||\boldA||_{HS} \le \sqrt{m} \, ||\boldA||_{op}
\end{equation}
for all $m \times m$ matrices $\boldA$.

The Hilbert--Schmidt norm is submultiplicative, hence
\begin{equation}\label{eq:HS-sub-multiplicative-1}
||\boldA^2||_{HS} \le ||\boldA||_{HS}^2
\end{equation}
for all $m \times m$ matrices $\boldA$. Inequality \eqref{eq:HS-sub-multiplicative-1} can be improved for skew-symmetric matrices. Indeed, if $\boldA$ is skew-symmetric ($\boldA^T = -\boldA$), then
\begin{equation}\label{eq:HS-sub-multiplicative-2}
\frac1{\sqrt{m}} ||\boldA||_{HS}^2 \le ||\boldA^2||_{HS} \le \frac1{\sqrt2} ||\boldA||_{HS}^2 \, .
\end{equation}
Inequality \eqref{eq:HS-sub-multiplicative-2} can be established by appealing to the canonical block diagonal representation of skew-symmetric matrices up to orthogonal equivalence. Details can be found, for instance, at \cite{mathstackexchange}. More precisely, if $\boldA$ is a skew-symmetric $m \times m$ matrix, then $\boldA$ is orthogonally equivalent to a block diagonal matrix of the form
$$
\begin{pmatrix} 0 & \lambda_1 \\ -\lambda_1 & 0 \end{pmatrix} \oplus \cdots \oplus \begin{pmatrix} 0 & \lambda_\ell \\ -\lambda_\ell & 0 \end{pmatrix}
$$
if $m = 2 \ell$ is even, or
$$
0 \oplus \begin{pmatrix} 0 & \lambda_1 \\ -\lambda_1 & 0 \end{pmatrix} \oplus \cdots \oplus \begin{pmatrix} 0 & \lambda_\ell \\ -\lambda_\ell & 0 \end{pmatrix}
$$
if $m = 2 \ell + 1$ is odd. In either case, we deduce that
$$
||\boldA||_{HS}^2 = \trace(-\boldA^2) = 2\sum_{j=1}^\ell \lambda_j^2
$$
and
$$
||\boldA^2||_{HS}^2 = \trace(\boldA^4) = 2\sum_{j=1}^\ell \lambda_j^4 \, .
$$
The inequalities \eqref{eq:HS-sub-multiplicative-2} easily follow. The inequalities in \eqref{eq:HS-sub-multiplicative-2} will play a role several times in this paper. See, for instance, the proof of Proposition \ref{prop:delta_0-characterization-of-nascent-H-type-groups} as well as our computation of the H-type deviation of free Carnot groups of step two in Example \ref{ex:free-group}.

\subsection{Stratified Lie groups}

Let $\G$ be a Lie group with Lie algebra $\fg$. We say that $\G$ is a {\it step two stratified Lie group} if $\fg$ admits a vector space decomposition, $\fg = \fv_1 \oplus \fv_2$ so that
\begin{equation}\label{eq:lie-bracket-relation}
[\fv_1,\fv_1]=\fv_2
\end{equation} 
and $[\fg,\fv_2] = 0$. Hence $\fg$ (and hence also $\G$) is nilpotent of step two. We denote by $m_1$, respectively $m_2$, the dimension of $\fv_1$, respectively $\fv_2$. We also abbreviate $m := m_1$. We may identify the Lie algebra with a Euclidean space,
$$
\fg = \fv_1 \oplus \fv_2 \simeq \R^{m_1 + m_2} =: \R^N.
$$
Since $\G$ is nilpotent, the exponential map $\exp:\fg \to \G$ is a diffeomorphism, and $\G$ itself can also be identified with the space $\R^N$. We denote the identity element of $\G$ by $0$.

As usual, we identify elements of $\fg$ with left invariant vector fields on $\G$. Under this identification, $\fv_1$ induces the {\it horizontal tangent bundle} of $\G$, denoted $H\G$. The horizontal bundle $H\G$ satisfies the bracket-generating condition, and hence any two points in $\G$ can be joined by a piecewise smooth curve, almost all of whose tangent vectors lie in the horizontal tangent bundle $H\G$. Such a curve is called {\it horizontal}.

The {\it dilation} $\delta_\lambda$ with {\it scale factor} $\lambda > 0$ is defined first on the level of the Lie algebra $\fg$ by $\delta_\lambda(U+T) = \lambda U + \lambda^2 T$, where $U+T \in \fv_1 \oplus \fv_2$. We promote $\delta_\lambda$ to a self-map of $\G$ by conjugating with the exponential map: $\delta_\lambda \circ \exp = \exp \circ \delta_\lambda$. A real-valued function $u$ on $\G$ is said to be {\it $s$-homogeneous}, $s \in \R$, if $u \circ \delta_\lambda = \lambda^s \, u$ for all $\lambda>0$.

Let $\mu$ be Haar measure in $\G$; it coincides (up to a multiplicative constant) with Lebesgue measure $\cL^N$ in $\G\simeq\R^N$. The Jacobian determinant of $\delta_\lambda$ is constant and equal to $\lambda^Q$, where
\begin{equation}\label{eq:Q}
Q = m + 2m_2
\end{equation}
is the {\it homogeneous dimension} of $\G$.

\subsection{Stratified Lie groups with sub-Riemannian structure, a.k.a.\ Carnot groups} Let $\G$ be a step two stratified Lie group as above. Let $g_h$ be an inner product defined in the horizontal layer $\fv_1$. We extend $g_h$ to a family of inner products defined in the horizontal tangent bundle $H\G$ via left translation. We continue to denote the resulting family of inner products by $g_h$, and we call such a structure a {\it horizontal metric} on $\G$. We will also sometimes denote $g_h$ by $\langle \cdot,\cdot \rangle_h$, and we will also denote the corresponding norm by $||\cdot||_h$. 

The choice of a horizontal metric $g_h$ equips $\G$ with a sub-Riemannian structure. To wit, the {\it sub-Riemannian distance} between points $\ttg$ and $\tth$ in $\G$ is given by
$$
d_{cc}(\ttg,\tth) = \inf_\gamma \int_0^1 ||\gamma'(s)||_h \, ds
$$
where the infimum is taken over all piecewise smooth horizontal curves $\gamma$ which join $\ttg$ to $\tth$ in $\G$. We call such a group $\G$ equipped with a horizontal metric $g_h$ a {\it step two sub-Riemannian stratified Lie group}. We will also term such groups as {\it step two Carnot group}. 

\begin{remark}
Note that there is a slight subtlety in this terminology. In many places in the literature one finds the term {\bf Carnot group} used to signify a step two stratified Lie group (as described in the previous subsection) which {\bf may} be endowed with a sub-Riemannian structure via selection of a horizontal metric. For us, the term {\it Carnot group} will refer to a stratified Lie group for which such a horizontal metric {\bf has} been fixed.
\end{remark}

In each Carnot group $\G$ and for $\lambda>0$, the dilation $\delta_\lambda:\G\to\G$ scales the metric $d_{cc}$ by a factor of $\lambda$: $d_{cc}(\delta_\lambda(\ttg),\delta_\lambda(\tth)) = \lambda \, d_{cc}(\ttg,\tth)$, $\ttg,\tth \in \G$. It is well known that the Hausdorff dimension of the metric space $(\G,d_{cc})$ is the homogeneous dimension $Q$ defined in \eqref{eq:Q}.

A mapping $F: \G \to \G'$ between step two Carnot groups is called a {\it Carnot group isomorphism} if $F$ is an isomorphism of stratified Lie groups and $F^*g_h' = g_h$, where $g_h$ (respectively, $g_h'$) denotes the horizontal metric in $\G$ (respectively, in $\G'$).

\subsection{Carnot groups with Riemannian structure} Let $\G$ be a step two Carnot group as defined in the previous subsection. Assume in addition that an inner product $g_v$ is defined in the vertical layer $\fv_2$ so that $\fv_1$ is $g$-orthogonal to $\fv_2$, where $g = g_h \oplus g_v$. Via left invariance, $g$ induces a family of inner products in the entire tangent bundle $T\G$. Equivalently, $g$ defines a Riemannian metric on $\G$. We say that $g$ is a {\it taming Riemannian metric} if it arises in the preceding fashion. We call $g_h$, respectively $g_v$, the {\it horizontal metric} and {\it vertical metric} associated to this Riemannian metric.

A map $F: \G \to \G'$ between step two Carnot groups with taming Riemannian metrics $g$ and $g'$ (respectively) is a {\it tamed Riemannian Carnot group isomorphism} if $F$ is an isomorphism of stratified Lie groups and $F^*g' = g$.

\subsection{Groups of Heisenberg type.} Let $\G$ be a step two Carnot group equipped with a taming Riemannian metric $g_h \oplus g_v$. We recall the definition of Kaplan's $J$ operator on $\G$.

\begin{definition}[Kaplan \cite{kap:h-type}]\label{def:j-operator}
We define a mapping $\boldJ:\fv_2 \to \End(\fv_1)$ via the identity
\begin{equation}\label{eq:j-operator}
g_h(\boldJ_T(U),V) = g_v([U,V],T)
\end{equation}
valid for all $U,V \in \fv_1$ and $T \in \fv_2$.
\end{definition}

For each $T \in \fv_2$, $\boldJ_T$ is a skew-symmetric endomorphism of the inner product space $(\fv_1,g_h)$. We emphasize that the map $\boldJ$ depends on the choice of the inner product $g_v$ on the vertical complement $\fv_2$ inside $\fg$. At the risk of some lack of precision, we will not consistently indicate this dependence explicitly in the notation. When it is necessary to do so, we will modify the notation by writing
\begin{equation}\label{eq:J-dependence-on-g-v}
\boldJ = \boldJ^{[g_v]}.
\end{equation}

In any step two group, the $\boldJ$ operator satisfies the following orthogonality condition:
$$
g_h(\boldJ_T(U),U) = 0 \qquad \forall U \in \fv_1, T \in \fv_2
$$
This is an easy consequence of the definition. It follows that
$$
||a\,U + b\,\boldJ_T(U)||_h^2 = |a|^2 \, ||U||_h^2 + |b|^2 \, ||\boldJ_T(U)||_h^2 \qquad \forall U \in \fv_1, T \in \fv_2, a,b \in \R.
$$

\begin{definition}[Kaplan]
A step two Carnot group $\G$ with taming Riemannian metric $g_h \oplus g_v$ is an {\bf H-type (Heisenberg-type) group} if $\boldJ_T$ is an orthogonal self-map of $(\fv_1,g_h)$ for every $T \in \fv_2$ with $||T||_v = 1$. \end{definition}

In view of the skew-symmetry of $\boldJ_T$, it is easy to see that $\G$ is H-type if and only if the identity
\begin{equation}\label{eq:h-type-characterization}
\boldJ_T^2 = -||T||_v^2 \Id
\end{equation}
holds for all $T \in \fv_2$. Recall that $\Id$ denotes the identity operator acting in the horizontal space $\fv_1 \simeq \R^m$. It follows easily from \eqref{eq:h-type-characterization} that
\begin{equation}\label{eq:h-type-characterization-2}
||\boldJ_T(U)||_h^2 = ||T||_v^2 \, ||U||_h^2 \qquad \forall U \in \fv_1 , T \in \fv_2;
\end{equation}
by polarization we also have
\begin{equation}\label{eq:h-type-characterization-3}
g_h(\boldJ_T(U),\boldJ_T(V)) = ||T||_v^2 \, g_h(U,V) \qquad \forall U,V \in \fv_1, T \in \fv_2.
\end{equation}
In fact, either \eqref{eq:h-type-characterization-2} or \eqref{eq:h-type-characterization-3} could be taken as the definition of an H-type group. We refer to \cite[Chapter 18]{blu} for these and other properties of H-type groups.

\smallskip

Note that the definition of H-type group structure requires an {\it a priori} selection of a vertical metric $g_v$, while a sub-Riemannian structure on $\G$ is defined as soon as a horizontal metric $g_h$ has been fixed. 
In this paper, we consider H-type groups geometrically as a subset of the larger class of step two Carnot groups. The selection of a vertical metric complicates this analysis. To this end, we make the following definition for step two Carnot groups.

\begin{definition}
A step two Carnot group $\G$ with sub-Riemannian metric $g_h$ is said to be a {\bf nascent H-type group} if there exists a taming Riemannian metric $g = g_h \oplus g_v$ so that $(\G,g)$ is of H-type.
\end{definition}

The distinction between H-type groups and nascent H-type groups could be understood by analogy with the classical distinction between metric spaces and metrizable topological spaces.\footnote{With this analogy in mind, we might have elected to use the terminology {\it H-typable group} in place of {\it nascent H-type group}. We prefer the latter terminology over the awkward former one.} In this paper, we seek to recognize the nascent H-type groups among all step two Carnot groups.

\smallskip

In any step two stratified Lie group $\G$ we define analytic maps $\bx:\G \to \fv_1$ and $\bt:\G \to \fv_2$ via the identity
\begin{equation}\label{eq:exp}
\ttg = \exp( \bx(\ttg) + \bt(\ttg) ) \qquad \forall \ttg \in \G.
\end{equation}

\begin{definition}[Kaplan]\label{def:Kaplan-norm}
Assume that $\G$ is a step two Carnot group equipped with a taming Riemannian metric $g = g_h \oplus g_v$. For $\ttg \in \G$ we set
$$
a_v(\ttg) := ||\bx(\ttg)||_h^4 + 16||\bt(\ttg)||_v^2
$$
and
$$
N_v(\ttg) := a_v(\ttg)^{1/4}
$$
\end{definition}

The function $N_v:\G\to\R$ is {\bf Kaplan's quasinorm} on $\G$. Its dependence on the vertical metric $g_v$ is explicitly highlighted in our notation. However, due to the canonical choice of such metric in case $\G$ is an H-type group, we abbreviate $N_v =: N$ in that case.

\begin{example}[Heisenberg groups and anisotropic Heisenberg groups]\label{ex:heisenberg-example}
We emphasize yet again that the H-type condition requires a choice of a taming Riemannian metric. We illustrate with a discussion of Heisenberg and anisotropic Heisenberg groups. Recall that the Heisenberg group $\Heis^n$ is a step two stratified Lie group whose Lie algebra $\heis_n = \fv_1 \oplus \fv_2$ with $\dim \fv_1 = 2n$ and $\dim \fv_2 = 1$. The horizontal layer $\fv_1$ is spanned by vectors by 
\begin{equation}\label{eq:heis-basis}
X_1,\ldots,X_n,Y_1,\ldots,Y_n
\end{equation}
and the vertical layer $\fv_2$ is spanned by the single vector $T$, with
\begin{equation}\label{eq:heis-commutators}
[X_j,Y_j]=T \qquad \forall j=1,\ldots,n.
\end{equation}
Up to now we have only described a stratified Lie group structure on $\Heis^n$. We may further equip $\Heis^n$ with a sub-Riemannian (Carnot) structure by declaring the basis \eqref{eq:heis-basis} to be orthonormal for an inner product $g_h$ on $\fv_1$.

Finally, we may equip $\Heis^n$ with a Riemannian structure by extending $g_h$ to a metric $g = g_h\oplus g_v$ on all of $\heis_n$. Since $\fv_2$ is one-dimensional, the choice of such a vertical metric $g_v = g_v^{\lambda}$ is determined by one real parameter, namely, a constant $\lambda > 0$ so that $\lambda^{-1} T$ is a $g_v$-unit vector. The resulting one-parameter family of Riemannian metrics has been used extensively in the sub-Riemannian geometric analysis of the Heisenberg group. For instance, it is well known (going back to the work of Pansu) that the sub-Riemannian structure on $\Heis^n$ may be obtained via a suitable Gromov--Hausdorff limit of the degenerating Riemannian manifolds $(\Heis^n,g_v^{\lambda})$ as $\lambda \to \infty$.

For now, let us consider the vertical metric $g_v^{1}$. With respect to this choice, $T$ itself is a unit vector and Kaplan's J operator 
$$
\boldJ = \boldJ^{[g_v^{1}]}
$$
satisfies $\boldJ_T(X_j) = Y_j$ and $\boldJ_T(Y_j) = -X_j$ for all $j=1,\ldots,n$. It easily follows that the (isotropic) Heisenberg group $\Heis^n$ is endowed with the structure of an H-type group.

\smallskip
Now let $\bb = (b_1,\ldots,b_n)$ be a vector of positive real numbers. The {\bf anisotropic Heisenberg group $\Heis^n(\bb)$} is the step two Carnot group whose Lie algebra $\heis_n(\bb)$ has the same decomposition $\fv_1 \oplus \fv_2$, where $\fv_1$ has a basis as in \eqref{eq:heis-basis} and $\fv_2$ is spanned by a single element $T$. However, the commutator relationship \eqref{eq:heis-commutators} is replaced by
\begin{equation}\label{eq:anisotropic-heis-commutators}
[X_j,Y_j]=b_j \, T \qquad \forall j=1,\ldots,n.
\end{equation}
Hence $\Heis^n = \Heis^n(1,\ldots,1)$. Observe that any two anisotropic Heisenberg groups $\Heis^n(\bb)$ and $\Heis^n(\bb')$ are isomorphic as stratified Lie groups. In particular, all of these groups are isomorphic, as stratified Lie groups, to the usual Heisenberg group $\Heis^n$. Indeed, consider the map $F:\Heis^n(\bb) \to \Heis^n$ which is given, on the level of the Lie algebras, by
$$
{F_*(X_j) = \sqrt{b_j} X_j', \atop F_*(Y_j) = \sqrt{b_j} Y_j',} \qquad \qquad j=1,\ldots,n,
$$
and
$$
F_*(T) = T'.
$$
Here we denoted by $X_j,Y_j,T$ the standard stratified basis in $\Heis^n(\bb)$ and by $X_j',Y_j',T'$ the corresponding basis in $\Heis^n$. This map induces a stratified isomorphism between the two algebraic structures, which is easily seen to preserve the Lie bracket structure via the following computation:
$$
[F_*(X_j),F_*(Y_j)] = b_j [X_j',Y_j'] = b_j T' = F_*(b_j T) = F_*([X_j,Y_j]).
$$
Thus, when considered only as stratified Lie groups, all anisotropic Heisenberg groups of a fixed topological dimension are isomorphic. Note that in the case when all of the $b_j$s are equal to a fixed value $c$, we could alternatively have chosen $F$ so that $F_*$ maps $X_j$ to $X_j'$, maps $Y_j$ to $Y_j'$, and maps $T$ to $c ^{-1}T'$.

As before, we may fix a sub-Riemannian structure on $\Heis^n(\bb)$ by declaring the basis \eqref{eq:heis-basis} to be orthonormal with respect to a metric $g_h$ on the horizontal space. With respect to this specific choice of horizontal metric, we make the following claim.

\begin{proposition}
If $\bb$ is a nonzero multiple of $\bb'$, then $\Heis^n(\bb)$ and $\Heis^n(\bb')$ are isomorphic as Carnot groups.  
\end{proposition}

\begin{proof}
Assume that $\bb = c \bb'$ for some $c>0$. Define a stratified group isomorphism from $\Heis^n(\bb)$ to $\Heis^n(\bb')$, on the level of the Lie algebras, by
$$
{F_*(X_j) = X_j', \atop F_*(Y_j) = Y_j',} \qquad\qquad j=1,\ldots,n,
$$
and
$$
F_*(T) = c^{-1} T'.
$$
The fact that this isomorphism preserves the Lie bracket resides in the following computation:
$$
[F_*(X_j),F_*(Y_j)] = [X_j',Y_j'] = b_j' T' = c^{-1} b_j T' = F_*(b_j T) = F_*([X_j,Y_j]).
$$
Since the $X_j$s and $Y_j$s (respectively, the $X_j'$s and $Y_j'$s) form an orthonormal basis for $g_h$ (respectively $g_h'$), it follows from the definition of $F_*$ that $F^* g_h' = g_h$. Hence $F$ also preserves the sub-Riemannian structure, and makes $\Heis^n(\bb)$ and $\Heis^n(\bb')$ isomorphic as Carnot groups. 
\end{proof}

Finally, we assert that {\it the anisotropic Heisenberg group $\Heis^n(\bb)$ admits a taming Riemannian metric which makes it into an H-type group if and only if $b_1=b_2=\cdots=b_n$.} Indeed, if $b_j=\lambda^{-2}$ for all $j=1,\ldots,n$, then the vertical metric $g_v^{\lambda}$ (as defined above) induces an H-type structure on $\Heis^n(\bb)$. We postpone the proof of the converse direction to the following section, where we explicitly compute the deviation of anisotropic Heisenberg groups from the class of nascent H-type groups.
\end{example}

\section{H-type deviation of step two Carnot groups.} 

In this section we introduce the fundamental new concept underlying this work, the {\bf H-type deviation} of a step two Carnot group. We give basic properties, show that this quantity characterizes nascent H-type groups, and compute its value for several illustrative examples.

\subsection{H-type deviation: definition and basic properties}

\begin{definition}\label{def:HS-deviation}
Let $\G$ be a step two Carnot group. The {\bf Hilbert--Schmidt deviation of $\G$ from nascent H-type groups} (for short, the {\bf H-type deviation of $\G$}) is the quantity
\begin{equation}\label{eq:delta0}
\delta(\G) = \frac1{\sqrt{\dim\fv_1}} \, \inf_{g_v} \sup_{\substack{T \in \fv_2 \\ ||T||_v = 1}} ||(\boldJ_T^{[g_v]})^2 + \Id||_{HS} \, ,
\end{equation}
where the infimum is taken over all vertical metrics $g_v$. 
\end{definition}

Recall that $||\cdot||_{HS}$ denotes the Hilbert--Schmidt norm on matrices. We used this norm in Definition \ref{def:HS-deviation} for ease of computation. See, for instance, Examples \ref{ex:anisotropic-Heisenberg} and \ref{ex:free-group}. However, other matrix norms could be used in its place if desired.

Observe that $\delta(\G) < \delta$ if and only if there exists a taming Riemannian metric $g = g_h \oplus g_v$ on $\fg$  so that
\begin{equation}\label{eq:delta-alternate}
\norm{(\boldJ_T^{[g_v]})^2 + ||T||_v^2 \, \Id}_{HS} < \delta \, \sqrt{\dim\fv_1} \, ||T||_v^2 \qquad \forall T \in \fv_2.
\end{equation}
On occasion, we will want to consider the analogous quantity for a fixed vertical metric $g_v$. We therefore also introduce the notation
\begin{equation}\label{eq:delta0v}
\delta(\G,g_v) := \frac1{\sqrt{\dim \fv_1}} \, \sup_{\substack{T \\ ||T||_v = 1}} \norm{(\boldJ_T^{[g_v]})^2 + \Id}_{HS}.
\end{equation}
Thus 
$$
\delta(\G) = \inf_{g_v} \delta(\G,g_v).
$$ 

The relevance of the H-type deviation is highlighted by the following proposition. Proposition \ref{prop:delta_0-characterization-of-nascent-H-type-groups} plays a key role throughout this paper as the critical final step in each of our new characterizations for the class of H-type groups.

\begin{proposition}\label{prop:delta_0-characterization-of-nascent-H-type-groups}
Let $\G$ be a step two Carnot group. Then $\delta(\G) = 0$ if and only if $\G$ is a nascent H-type group.
\end{proposition}

\begin{proof}
If $\G$ is a nascent H-type group, then $\G$ supports a vertical metric $g_v$ so that $(\boldJ_T^{[g_v]})^2 = -\Id$ for every $T \in \fv_2$ with $||T||_v = 1$. Hence $\delta(\G) = 0$.

Now suppose that $\delta(\G) =0$. Then there exist vertical metrics $(g_{v,j})_{j \ge 1}$ so that
\begin{equation}\label{eq:prop-1}
\delta(\G,g_{v,j}) < \frac1j \qquad \forall j \ge 1.
\end{equation}

Fix $T \in \fv_2$. Choose $U_T,V_T \in \fv_1$ so that $[U_T,V_T]=T$ (cf.\ \eqref{eq:lie-bracket-relation}). Then
\begin{equation*}\begin{split}
||T||_{{v,j}}^2 
&= \langle [U_T,V_T], T \rangle_{{v,j}} \\
&= \langle \boldJ_T^{[g_{v,j}]}(U_T),V_T \rangle_h \\
\intertext{by \eqref{eq:j-operator}}
&\le \bigl|\bigl|\boldJ_T^{[g_{v,j}]}\bigr|\bigr|_{op} ||U_T||_h ||V_T||_h \\
&\le  \bigl|\bigl|\boldJ_T^{[g_{v,j}]}\bigr|\bigr|_{HS} ||U_T||_h ||V_T||_h \\
\intertext{by \eqref{eq:op-HS}}
&= \bigl|\bigl|\boldJ_{T/||T||_{v,j}}^{[g_{v,j}]}\bigr|\bigr|_{HS} ||T||_{{v,j}} \, ||U_T||_h ||V_T||_h \\
&\le m^{1/4} \, \bigl|\bigl|(\boldJ_{T/||T||_{v,j}}^{[g_{v,j}]})^2\bigr|\bigr|_{HS}^{1/2} \, ||T||_{{v,j}} \, ||U_T||_h ||V_T||_h \\
\intertext{by \eqref{eq:HS-sub-multiplicative-2}}
&\le m^{1/4} \, \left( \sqrt{m} + \bigl|\bigl|(\boldJ_{T/||T||_{v,j}}^{[g_{v,j}]})^2 + \Id\bigr|\bigr|_{HS} \right)^{1/2} \, ||T||_{{v,j}} \, ||U_T||_h ||V_T||_h \\
\end{split}\end{equation*}
Hence
\begin{equation}\begin{split}\label{eq:gvj1.5}
||T||_{{v,j}} 
&\le \sqrt{m}  \, \left( 1 + \delta(\G,g_{v,j}) \right)^{1/2} \, ||U_T||_h ||V_T||_h \\
&< \sqrt{m} (1+\tfrac1j)^{1/2} \, ||U_T||_h ||V_T||_h \\
&\le \sqrt{2m} \, ||U_T||_h ||V_T||_h\, .
\end{split}\end{equation}
for all $T \in \fv_2$. Using Cauchy--Schwarz we now conclude that
\begin{equation}\begin{split}\label{eq:gvj2}
|g_{v,j}(S,T)| \le 2m \, ||U_T||_h ||V_T||_h ||U_S||_h ||V_S||_h 
\end{split}\end{equation}
for all $S,T \in \fv_2$. 

Now fix $S$ and $T$ in $\fv_2$. The uniform estimate \eqref{eq:gvj2} ensures that the sequence $(g_{v,j}(S,T))$ subconverges to a value $g_{v,\infty}(S,T)$ which again satisfies 
$$
|g_{v,\infty}(S,T)| \le 2m \, ||U_T||_h ||V_T||_h ||U_S||_h ||V_S||_h \, .
$$
Using the separability of $\fv_2$, diagonalization, and a continuity argument, we conclude that $g_{v,\infty}$ defines a symmetric bilinear form on $\fv_2$ and hence induces a taming Riemannian metric $g_h \oplus g_{v,\infty}$ on $\fg$. Furthermore, $||T||_{v,j} \to ||T||_{v,\infty}$ and $\boldJ_T^{[g_{v,j}]} \to \boldJ_T^{[g_{v,\infty}]}$ for any $T \in \fv_2$. 

It follows from \eqref{eq:delta-alternate}, \eqref{eq:prop-1}, and \eqref{eq:gvj1.5} that
\begin{equation}\begin{split}\label{eq:gvj3}
\bigl|\bigl|(\boldJ_T^{[g_{v,j}]})^2 + ||T||_{v,j}^2 \, \Id\bigr|\bigr|_{HS} < \frac{2m^{3/2}}j \, ||U_T||_h^2 ||V_T||_h^2
\end{split}\end{equation}
for all $T \in \fv_2$. Letting $j \to \infty$ in \eqref{eq:gvj3}, we conclude that $\norm{ (\boldJ_T^{[g_{v,\infty}]})^2 + ||T||_{v,\infty}^2 \, \Id}_{HS} = 0$ for all $T$.  We conclude that $ (\boldJ_T^{[g_{v,\infty}]})^2= - ||T||_{v,\infty}^2 \, \Id$ for all $T \in \fv_2$, i.e., $(\G,g_h \oplus g_{v,\infty})$ is an H-type group. This completes the proof.
\end{proof}

\begin{remark}
A similar proof shows that the infimum in the definition of $\delta(\G)$ is always attained by some vertical metric. We omit the details.
\end{remark}

\subsection{H-type deviation: examples}\label{sec:examples}

We now compute the H-type deviation for several classes of step two Carnot groups $\G$.

\begin{example}[Anisotropic Heisenberg groups]\label{ex:anisotropic-Heisenberg}
Let $\G = \Heis^n(\bb)$ for some $\bb=(b_1,\ldots,b_n)$. For $\lambda>0$, let $g_v^\lambda$ denote the vertical metric on $\G$ for which $\lambda^{-1} T$ is a unit vector. For $p \ge 1$, let
$$
||\bb||_{p} := \left( \tfrac1n \sum_{j=1}^n b_j^p \right)^{1/p}
$$ 
denote the $\ell^p$ norm of $\bb$ with respect to the uniform probability measure on $\{1,\ldots,n\}$.

We will show that
\begin{equation}\label{eq:delta0anisotropic-Heisenberg}
\delta(\Heis^n(\bb)) = \delta(\Heis^n(\bb),g_v^{\lambda_0}) = \sqrt{1- \left(\frac{||\bb||_2}{||\bb||_4}\right)^4} \,,
\end{equation}
where $\lambda_0 = ||\bb||_{2}/(||\bb||_4)^2$. Observe that $||\bb||_{2} \le ||\bb||_{4}$ by the Cauchy--Schwarz inequality, with equality if and only if $\bb$ is a constant vector. Hence $\delta(\Heis^n(\bb)) = 0$ if and only if $\bb$ is a constant vector. As noted above, this condition precisely characterizes the case when $\Heis^n(\bb)$ is Carnot group isomorphic to the Heisenberg group $\Heis^n$.

Using the previously mentioned description of vertical metrics on $\Heis^n(\bb)$, we find that
\begin{equation*}\begin{split}
\delta(\Heis^n(\bb)) 
&= \frac1{\sqrt{2n}} \inf_{\lambda>0} || (\boldJ_{\lambda^{-1} T}^{[g_v^\lambda]})^2 + \Id||_{HS} \\
&= \frac1{\sqrt{2n}} \inf_{\lambda>0} || \lambda^{-2} (\boldJ_{T}^{[g_v^\lambda]})^2 + \Id||_{HS} \\
&= \inf_{\lambda>0} \sqrt{1-2\lambda^2 (||\bb||_2)^2 + \lambda^4 (||\bb||_4)^4} \, .
\end{split}\end{equation*}
The function $\lambda \mapsto F(\lambda) := 1-2\lambda^2 (||\bb||_2)^2 + \lambda^4 (||\bb||_4)^4$ is minimized for $\lambda = \lambda_0$, and $F(\lambda_0) = \sqrt{1- (||\bb||_{2} / ||\bb||_{4})^4}$. This completes the proof.
\end{example}

\begin{remark}
The {\bf generalized H-type groups}, introduced by Barilari and Rizzi \cite{br:generalized-H-type} include all anisotropic Heisenberg groups. A step two Carnot group $\G$ is said to be a {\bf generalized H-type group} if it carries a taming Riemannian metric $g = g_h \oplus g_v$ and there exists a symmetric linear operator $\boldS$ acting on $\fv_1$ so that $\boldJ_T^2 = -||T||_v^2 \, \boldS$ for all $T \in \fv_2$. It is immediate to see that for any generalized H-type group $\G$, we have 
$$
\delta(\G,g_v) = \frac{||\Id - \boldS||_{HS}}{||\Id||_{HS}} \, .
$$
Note that the class of generalized H-type groups is strictly larger than the class of H-type groups. For example, any step two Carnot group with one-dimensional vertical layer $\fv_2$ can be equipped with the structure of a generalized H-type group.
\end{remark}

\begin{example}[Free Carnot groups of step two]\label{ex:free-group}
Let $\G = \F_{2,m}$ denote the free Carnot group of step two and rank $m$. We first describe the structure of this group as a stratified Lie group. The Lie algebra admits the decomposition $\ff_{2,m} = \fv_1 \oplus \fv_2$ where $\fv_2 = \Lambda^2 \fv_1$ is the second exterior power of $\fv_1$. We select a basis $X_1,\ldots,X_m$ for $\fv_1$ and a basis $T_{ij}$, $1\le i < j \le m$, for $\fv_2$ so that $[X_i,X_j] = T_{ij}$ for all $i<j$. We emphasize that no other bracket relations are imposed beyond those mandated by the Lie algebraic structure of $\fg$. The horizontal layer $\fv_1$ has dimension $m$ while the vertical layer $\fv_2$ has dimension
$$
m_2 = \binom{m}{2}.
$$
We will consistently enumerate coordinates in $\fv_2$ using pairs of coordinates in $\fv_1$. We interpret the elements of $\fv_2$ in two ways:
\begin{itemize}
\item as a vector $\boldt = (t_{ij})$ in $\R^{\binom{m}{2}}$, indexed by pairs $(i,j)$ with $1\le i<j\le m$, and
\item as a skew-symmetric $m\times m$ matrix $\boldX = (x_{ij})$, where we identify the coordinate $t_{ij}$ with the $(i,j)$th entry $x_{ij}$ and extend this assignment uniquely as a skew-symmetric matrix.
\end{itemize}
We denote by $\boldX_\boldt$ the skew-symmetric $m\times m$ matrix corresponding to the vector $\boldt \in \R^{\binom{m}{2}}$. Observe that
$$
||\boldX_\boldt||_{HS} = \sqrt{2} \, |\boldt|_2.
$$

Next, we introduce a Carnot group structure in $\F_{2,m}$ by declaring $X_1,\ldots,X_m$ to be an orthonormal basis for a metric $g_h$ on the horizontal layer $\fv_1$. The homogeneous dimension of this Carnot group is $Q = m + 2 m_2 = m^2$.

When $m=2$, the group $\F_{2,2}$ is isomorphic to the first Heisenberg group, which is an H-type group. Hence $\delta(\F_{2,2}) = 0$. We will prove that
\begin{equation}
\delta(\F_{2,m}) = \sqrt{\frac{m-2}{m}} \qquad \forall \, m\ge 2.
\end{equation}
Fix $m \ge 3$. We will first show that
\begin{equation}\label{eq:delta0free2step}
\delta(\F_{2,m}) \le \sqrt{\frac{m-2}{m}} \, .
\end{equation}
To this end, we select the taming Riemannian metric $g=g_h \oplus g_v$ for which the standard vectors $T_{ij}$, $1\le i<j\le m$, are $g_v$-orthonormal. We find that
$$
\delta(\F_{2,m}) \le \delta(\F_{2,m},g_v) = \frac1{\sqrt{m}} \sup_{\boldt} || \boldJ_{T(\boldt)}^2 + \Id ||_{HS} \, ,
$$
where $T(\boldt) = \sum_{i<j} t_{ij} T_{ij}$ and the supremum is taken over all $\boldt = (t_{ij})_{i<j} \in \R^{\binom{m}{2}}$ with $|\boldt|_2 = 1$.

Observe that, for each $\boldt$ as above and each $i<j$, we have $g_h(\boldJ_{T(\boldt)}(X_i),X_j) = g_v(T_{ij},T(\boldt)) = t_{ij}$. In other words, the linear transformation $\boldJ_{T(\boldt)}$, expressed in the basis $X_1,\ldots,X_m$ for $\fv_1$, is given by multiplication by the skew symmetric matrix $\boldX_\boldt$. We conclude, under the assumption $|\boldt|_2= 1$, that
\begin{equation*}\begin{split}
\norm{ \boldJ_{T(\boldt)}^2 + \Id }_{HS}^2  
&= \trace( \boldJ_{T(\boldt)}^4 + 2 \boldJ_{T(\boldt)}^2 + \Id) \\
&= \trace( \boldX_\boldt^4 ) + 2 \trace(\boldX_\boldt^2) + m \\
& \le m - 2 \, ,
\end{split}\end{equation*}
where we used the facts that 
$$
\trace(\boldX_\boldt^2) = - ||\boldX_\boldt||_{HS}^2 = -2 |\boldt|_2^2 = -2
$$
and
$$
\trace( \boldX_\boldt^4 ) = || \boldX_\boldt^2 ||_{HS}^2 \le \frac12 ||\boldX_\boldt ||_{HS}^4 = 2 |\boldt|_2^4 = 2 \, ,
$$
cf.\ the inequality in \eqref{eq:HS-sub-multiplicative-2}.

\smallskip

Next, we prove that 
\begin{equation}\label{eq:delta0free2step2}
\delta(\F_{2,m}) \ge \sqrt{\frac{m-2}{m}}.
\end{equation}
To do this, we must bound $\delta(\F_{2,m},g_v)$ from below as $g_v$ ranges over all vertical metrics. Let $g_v$ be such a metric, and let $(S_{ij})$, $1\le i<j\le m$, be a $g_v$-orthonormal basis. Choose a matrix $\boldC = (c_{ij,k\ell}) \in \GL(m_2)$ which represents the standard basis $(T_{ij})$ in terms of the basis $(S_{ij})$. In other words, for each $1\le i<j\le m$,
$$
T_{ij} = \sum_{1\le k<\ell\le m} c_{ij,k\ell} \, S_{k\ell}.
$$
Now suppose that $T$ is a $g_v$-unit vector in $\fv_2$. Writing $T$ in the basis $(S_{k\ell})$, we identify $T$ with a vector $\boldt = (t_{k\ell})$ in $\R^{m_2}$, i.e., $t_{k\ell} = \langle T,S_{k\ell} \rangle_v$. Since $||T||_v = 1$, $\boldt$ is a Euclidean unit vector in $\R^{m_2}$. For $1\le i < j \le m$ we compute
$$
\langle \boldJ_T(X_i),X_j\rangle_h = \langle T,T_{ij} \rangle_v = \sum_{1\le k<\ell\le m} c_{ij,k\ell} \langle T,S_{k\ell} \rangle_v = (\boldC \cdot \boldt)_{ij}.
$$
It follows that the matrix representation of $\boldJ_T$ in the $g_h$-orthonormal basis $(X_i)$ is given by the skew-symmetric extension $\boldX_{\boldC \cdot \boldt}$ of the vector $\boldC \cdot \boldt \in \R^{\binom{m}{2}}$. In view of the preceding discussion, we conclude that
$$
\delta(\F_{2,m}) = \frac1{\sqrt{m}} \inf_{\boldC \in \GL(m_2)} \sup_{\substack{ \boldt \\ |\boldt|_2 = 1} } || \boldX_{c \cdot \boldt}^2 + \Id ||_{HS} \, .
$$
For each $\boldC \in \GL(m_2)$ we may choose $\lambda_0 = \lambda_0(\boldC)>0$ and a Euclidean unit vector $\boldt_0 = \boldt_0(\boldC) \in \R^{m_2}$ so that
$$
\boldC \cdot \boldt_0 = \lambda_0 \, \bolde_{12},
$$
where $\bolde_{12}$ denotes the first vector in the standard basis $(\bolde_{ij})_{i<j}$ for $\R^{\binom{m}{2}}$. Indeed, the choice
$$
\boldt_0 = \frac{\boldC^{-1} \cdot \bolde_{12}}{|\boldC^{-1} \cdot \bolde_{12}|_2} \qquad \mbox{and} \qquad \lambda_0 = \frac1{|\boldC^{-1} \cdot \bolde_{12}|_2} 
$$
works. Consequently,
\begin{equation*}\begin{split}
\delta(\F_{2,m}) 
&= \frac1{\sqrt{m}} \inf_{\boldC \in \GL(m_2)} \sup_{\substack{ \boldt \\ |\boldt|_2 = 1} } || \boldX_{\boldC \cdot \boldt}^2 + \Id ||_{HS} \\
&\ge \frac1{\sqrt{m}} \inf_{\boldC} || \boldX_{\boldC \cdot \boldt_0(c)}^2 + \Id ||_{HS}  \\
&=  \frac1{\sqrt{m}} \inf_{\boldC} || \lambda_0(\boldC)^2 \boldX_{{\bolde_{12}}}^2 + \Id ||_{HS}  \\
&\ge \frac1{\sqrt{m}} \inf_{\lambda>0} || \lambda^2 \boldX_{{\bolde_{12}}}^2 + \Id ||_{HS} \, .
\end{split}\end{equation*}
The matrix $\boldX_{{\bolde_{12}}}$ is a block diagonal matrix with one $2\times 2$ block of the form $\tiny \begin{pmatrix} 0 & 1 \\ -1 & 0 \end{pmatrix}$ and one $(m-2)\times(m-2)$ block consisting of all zeros. Hence
\begin{equation*}\begin{split}
|| \lambda^2 \boldX_{{\bolde_1}}^2 + \Id ||_{HS}^2 
&= \trace(\Id + 2 \lambda^2 \boldX_{{\bolde_1}}^2 + \lambda^4 \boldX_{{\bolde_1}}^4 ) \\
&= m - 4 \lambda^2 + 2 \lambda^4 \\
&= (m-2) + 2(1-\lambda^2)^2 \, .
\end{split}\end{equation*}
In conclusion, we obtain $\delta(\F_{2,m}) \ge \sqrt{(m-2)/m}$ as desired.
\end{example}

Lest the reader find the previous two examples too straightforward, we conclude with two examples demonstrating that the situation is more complicated than one might expect.

\begin{example}\label{ex:challenge-1}
For any $n \ge 2$ and $\eps \ge 0$, we define a step two stratified Lie group $\G^n_\epsilon$ of rank $2n$. The group has topological dimension $2n+2$ and two-dimensional center if $\eps>0$, and has topological dimension $2n+1$ and one-dimensional center if $\eps=0$. Specifically, the Lie algebra $\fg^n_\epsilon = \fv_1 \oplus \fv_2$ where $\fv_1$ is spanned by $2n$ elements $X_1,Y_1,\ldots,X_n,Y_n$ and
$$
\fv_2 = \begin{cases} \spa\{T\} & \epsilon = 0 \\ \spa\{T,U\} & \epsilon>0. \end{cases}
$$
The nontrivial Lie bracket relations in $\fv_\epsilon^n$ are 
$$
[X_j,Y_j] = T \quad \mbox{for $j=1,\ldots,n$} \qquad \qquad  \mbox{and} \qquad [X_1,X_2] = \epsilon U.
$$
Note that the resulting Lie algebra structure is well-defined regardless of whether $\epsilon = 0$ or $\epsilon$ is positive. It is easy to see that all of the groups $\G_\eps^n$ (for fixed $n$) are mutually isomorphic as stratified Lie groups. When $n=2$ and $\epsilon>0$, the group $\G^2_\epsilon$ is isomorphic (as a stratified Lie group) to the group labeled $N_{6,3,5}$ in Le Donne and Tripaldi's `Cornucopia of Carnot groups' \cite[p.\ 62]{ldt:cornucopia}.

For any $\eps\ge0$ we equip $\G^n_\eps$ with the structure of a Carnot group by declaring the basis 
\begin{equation}\label{eq:onb}
X_1,Y_1,\ldots,X_n,Y_n
\end{equation}
to be orthonormal for a horizontal metric $g_1$. When $\eps = 0$ the group $\G^n_0$ is isomorphic (as a Carnot group) to the Heisenberg group $\Heis^n$, and hence
$$
\delta(\G_0^n) = 0.
$$
In the limit as $\epsilon\to 0$ the Lie bracket relations converge to the corresponding relations in $\Heis^n$, with the two-dimensional center collapsing down to a one-dimensional center in the limit. Naively, one might expect that the values of the H-type deviation would follow suit. However, we will show that
\begin{equation}\label{eq:delta-G-epsilon-n}
\delta(\G^n_\eps) \ge \sqrt{1-\tfrac1n} \qquad \mbox{for all $n \ge 2$ and $\eps > 0$.}
\end{equation}
In fact, all of the groups $\G_\eps^n$ (for fixed $n$) are mutually equivalent as Carnot groups. This is true since it is possible to find a stratified Lie group isomorphism between any two such groups which preserves all of the elements in the $g_h$-orthonormal basis \eqref{eq:onb}. Thus it suffices to consider $\eps=1$ in what follows, and the preceding naive assumption is clearly false. However, in order to clarify the relationship between this example and Example \ref{ex:challenge-2} we will continue to work with a general positive $\eps$ for the remainder of this example.

Consequently, we fix $\eps>0$ and consider a vertical metric $g_v$ on $\fv_2$ with orthonormal basis $S_1$ and $S_2$. As in previous examples, we identify $g_v$ with a matrix 
\begin{equation}\label{eq:C-in-GL-2}
\boldC = \begin{pmatrix} a & b \\ c & d \end{pmatrix} \in \GL(2)
\end{equation}
by writing $T,U$ in the basis $S_1,S_2$:
\begin{eqnarray*}
T &= a S_1 + b S_2, \\
U &= c S_1 + d S_2.
\end{eqnarray*}
The second layer $\fv_2$ consists of all vectors $S = t_1 S_1 + t_2 S_2$ where $\boldt = (t_1,t_2)$ ranges over $\R^2$, and $g_v(S,S) = |\boldt|_2^2$. We compute
$$
\langle \boldJ_S(X_j),Y_j \rangle_h = \langle T,S \rangle_v = at_1 + bt_2 = (\boldC \cdot \boldt)_1 \qquad j=1,\ldots,n,
$$
and
$$
\langle \boldJ_S(X_1),X_2 \rangle_h = \eps \, \langle U,S \rangle_v = \eps(ct_1 + dt_2) = \eps (\boldC \cdot \boldt)_2 \, .
$$
Hence the matrix representation of $\boldJ_S^{[g_v]}$ in the $g_h$-orthonormal basis $X_1,Y_1,\ldots,X_n,Y_n$ is
$$
\left( 
\mbox{ 
\begin{tabular}{c|c}
\begin{tabular}{c|c}
$\begin{matrix} 0 & \eps (\boldC\cdot\boldt)_2 \\ - \eps (\boldC\cdot\boldt)_2 & 0 \end{matrix}$ & $\boldO_{2\times(n-2)}$ \\ \hline \\ $\boldO_{(n-2)\times 2}$ & $\boldO_{(n-2)\times(n-2)}$ 
\end{tabular} 
& $(\boldC\cdot\boldt)_1 \, \Id_n$ \\ \hline \\ $-(\boldC\cdot\boldt)_1 \, \Id_n$ & $\boldO_{n\times n}$
\end{tabular} 
}
\right)
$$
where $\boldO_{m\times n}$ denotes an $m\times n$ block consisting entirely of zeros. Writing $x = (\boldC\cdot\boldt)_1$ and $y = (\boldC\cdot\boldt)_2$, we find that the representation of $(\boldJ_S^{[g_v]})^2 + \Id$ is
$$
\left( 
\mbox{ 
\begin{tabular}{c|c|c|c}
$(1 - x^2 - \eps^2 y^2) \Id_2$ & $\boldO_{2\times(n-2)}$ & \mbox{$\begin{matrix} 0 & \eps x y \\ -\eps x y & 0 \end{matrix}$} & $\boldO_{2\times(n-2)}$ \\
\hline
$\boldO_{(n-2)\times 2}$ & $(1 - x^2) \Id_{n-2}$ & $\boldO_{(n-2)\times 2}$ & $\boldO_{(n-2)\times(n-2)}$ \\
\hline
\mbox{$\begin{matrix} 0 & -\eps x y \\ \eps x y & 0 \end{matrix}$} & $\boldO_{2\times(n-2)}$ & $(1 - x^2) \Id_{2}$ & $\boldO_{2\times(n-2)}$ \\
\hline
$\boldO_{(n-2)\times 2}$ & $\boldO_{(n-2)\times(n-2)}$ & $\boldO_{(n-2)\times 2}$ & $(1 - x^2) \Id_{n-2}$ \\
\end{tabular}
}
\right)
$$
and thus
\begin{equation*}\begin{split}
||(\boldJ_S^{[g_v]})^2 + \Id||_{HS}^2 &= 2(1-x^2-\eps^2 y^2)^2 + (2n-2)(1-x^2)^2 + 4\eps^2 x^2 y^2  \\
&= (2n)(1-x^2)^2 + 4 (2x^2-1)y^2 \eps^2 + 2 y^4 \eps^4 \, .
\end{split}\end{equation*}
We conclude that
$$
\delta(\G_\epsilon^n) = \inf_{\boldC \in \GL(2)} \sup_{\substack{\boldt \\ |\boldt|_2 = 1}} \sqrt{(1-x^2)^2 + \tfrac2n(2x^2-1)y^2 \eps^2 + \tfrac1n y^4 \eps^4} \qquad {x = (\boldC\cdot\boldt)_1, \atop y = (\boldC\cdot\boldt)_2 \, .}
$$
We express the rows of $\boldC$ in polar coordinates: $(a,b) = (p\cos\alpha,p\sin\alpha)$ and $(c,d) = (q\cos\beta,q\sin\beta)$ and we write $\boldt = (t_1,t_2) = (\cos\theta,\sin\theta)$. Then
$$
\delta(\G_\epsilon^n) = \inf \sup_\theta \sqrt{(1-x^2)^2 + \tfrac2n(2x^2-1)y^2 \eps^2 + \tfrac1n y^4 \eps^4} \qquad {x = p\cos(\theta-\alpha) \atop y=q\cos(\theta-\beta)} \, ,
$$
where the infimum is taken over all positive $p$ and $q$ and all $\alpha \ne \beta$ modulo $2\pi$. 
Replacing $\theta$ and $\alpha$ by $\theta+\beta$ and $\alpha+\beta$ respectively gives
\begin{equation}\label{eq:dGen}
\delta(\G_\epsilon^n) = \inf \sup_\theta \sqrt{(1-p^2\cos^2(\theta-\alpha))^2 + \tfrac2n(2p^2\cos^2(\theta-\alpha)-1)\cos^2(\theta) (q\eps)^2 + \tfrac1n \cos^4(\theta) (q\eps)^4},
\end{equation}
where the infimum is over all positive $p$ and $q$ and all $\alpha\ne 0$ modulo $2\pi$. Note that \eqref{eq:dGen} again shows that the Carnot group structure does not depend on $\eps$; the value of $\delta(\G_\epsilon^n)$ is the same for all choices of $\eps>0$ since the infimum is over all positive $q$.

For fixed $p>0$, $q>0$, and $\alpha \ne 0$, the supremum over $\theta$ in \eqref{eq:dGen} is no less than the value of the argument when $\theta = \alpha + \tfrac\pi2$. Therefore
\begin{equation*}\begin{split}
\delta(\G_\epsilon^n) 
&\ge \inf_{q>0,\alpha\ne 0} \sqrt{1 - \tfrac2n\sin^2(\alpha) (q\eps)^2 + \tfrac1n \sin^4(\alpha) (q\eps)^4} \\
&= \inf_{q>0,\alpha\ne 0} \sqrt{(1-\tfrac1n) + \tfrac1n(1 - \sin^2(\alpha) (q\eps)^2)^2} = \sqrt{1-\tfrac1n}.
\end{split}\end{equation*}
This completes the proof of \eqref{eq:delta-G-epsilon-n}. In particular, not only does the H-type deviation of $\G^n_\epsilon$ fail to converge to that of $\G^n_0$ as $\epsilon \to 0$, but in addition, $\delta(\G^n_1) \to 1$ as $n \to \infty$.
\end{example}

\begin{example}\label{ex:challenge-2}
Now consider the group $\bar\G_\eps^n$, $\eps>0$, defined as follows. The Lie algebra $\bar\fg_\eps^n = \fv_1 \oplus \fv_2$ where $\fv_1$ is spanned by $2n$ elements $X_1,Y_1,\ldots,X_n,Y_n$ and $\fv_2$ is spanned by a single element $T$. The nontrivial bracket relations in $\bar\fg_\eps^n$ are 
$$
[X_j,Y_j] = T \quad \mbox{for $j=1,\ldots,n$} \qquad \qquad  \mbox{and} \qquad [X_1,X_2] = \epsilon T.
$$
The group $\bar\G_\eps^n$ is isomorphic, as a stratified Lie group, to the Heisenberg group $\Heis^n$. Indeed, the new basis 
$$
\bar{X}_1 = X_1, \quad \bar{Y}_1 = Y_1, \quad \bar{X}_2 = X_2 - \eps Y_1, \quad \bar{Y}_2 = Y_2, \quad \dots \quad \bar{X}_n = X_n, \quad \bar{Y}_n = Y_n, \quad \mbox{and} \quad \bar{T} = T
$$ 
verifies the usual Heisenberg bracket relations. We equip $\bar\G_\eps^n$ with the structure of a Carnot group by declaring the basis $\bar{X}_1,\bar{Y}_1,\ldots,\bar{X}_n,\bar{Y}_n$ to be $g_h$-orthonormal. Note that the preceding stratified Lie group isomorphism is {\bf not} an isomorphism of Carnot groups with respect to this choice. We claim that
$$
\delta(\bar\G_\eps^n) = O(\eps),
$$
in particular, $\delta(\bar\G_\eps^n) \to 0$ as $\eps \to 0$. As in an earlier example, we note that vertical metrics on $\bar\G_\eps^n$ are indexed by a positive real parameter $\lambda$. Denote by $g_v^\lambda$ the vertical metric for which $\lambda^{-1} T$ is a $g_v$-unit vector. The matrix representation of $\boldJ_{T}^{[g_v]}$ in the $g_h$-orthonormal basis $X_1,Y_1,\ldots,X_n,Y_n$ is
$$
\left( 
\mbox{ 
\begin{tabular}{c|c}
\begin{tabular}{c|c}
$\begin{matrix} 0 & \eps \lambda^2 \\ - \eps \lambda^2 & 0 \end{matrix}$ & $\boldO_{2\times(n-2)}$ \\ \hline \\ $\boldO_{(n-2)\times 2}$ & $\boldO_{(n-2)\times(n-2)}$ 
\end{tabular} 
& $\lambda^2 \, \Id_n$ \\ \hline \\ $- \lambda^2 \, \Id_n$ & $\boldO_{n\times n}$
\end{tabular} 
}
\right)
$$
and hence
\begin{equation*}\begin{split}
||(\boldJ_{\pm\lambda^{-1}T}^{[g_v^\lambda]})^2 + \Id||_{HS}^2
&= ||\lambda^{-2} \, (\boldJ_{T}^{[g_v^\lambda]})^2 + \Id||_{HS}^2 \\
&= (2n)(1-\lambda^2)^2 + 4(2\lambda^2-1)\lambda^2 \eps^2 + 2 \lambda^4 \eps^4 =: F(\lambda).
\end{split}\end{equation*}
The minimum of $F(\lambda)$ over $\lambda>0$ is achieved when $\lambda^2 = \lambda_0^2 := (n+\eps^2)/(n+4\eps^2+\eps^4)$ and
$$
\delta(\bar\G_\eps^n) = \frac1{\sqrt{2n}} \sqrt{F(\lambda_0)} = \eps \, \sqrt{\frac{2+\tfrac{n-1}{n} \eps^2}{n+4\eps^2+\eps^4}} = O(\eps).
$$
Moreover, since $\eps \mapsto \delta(\bar\G_\eps^n)$ is strictly increasing (for fixed $n$) and the H-type deviation depends only on the Carnot group structure, we also deduce that no pair of these groups, $\G_{\eps_1}^n$ and $\G_{\eps_2}^n$, are isomorphic as Carnot groups. (The latter fact is also easy to see directly from the definition of this group.) By way of contrast, all of the groups $\G_\eps^n$, $\eps>0$, considered in Example \ref{ex:challenge-1} are isomorphic as Carnot groups.
\end{example}

\section{Quantitative deviation of step two Carnot groups from H-type groups}\label{sec:Htype}

Let $\G$ be a step two Carnot group, with Lie algebra $\fg = \fv_1 \oplus \fv_2$. We fix bases $X_1,\ldots,X_m$ and $T_1,\ldots,T_{m_2}$ for $\fv_1$ and $\fv_2$. We identify $\G$ with $\R^{m+m_2} \simeq \fg$ via the exponential map and introduce coordinates $(x_1,\ldots,x_m,t_1,\ldots,t_{m_2})$ in $\G$ via the assignment $\boldx(g) = \sum_j x_j X_j$ and $\boldt(g) = \sum_{q=1}^{m_2} t_q T_q$. With respect to these coordinates, the left-invariant horizontal vector fields $X_i$ take the following form:
$$
X_i = \frac{\partial}{\partial x_i} - \frac12 \sum_{q=1}^{m_2} \sum_{j=1}^m \, b_{ij}^q \, x_j \, \frac{\partial}{\partial t_q} \, .
$$
The coefficients $b_{ij}^q$, $1\le i,j\le m$, $1\le q \le m_2$, are the {\bf structure constants} of the Lie algebra. For each fixed $q$, the matrix $\boldB^q := (b_{ij}^q)$ is a skew-symmetric $m \times m$ matrix. We record the identity
$$
[X_i,X_j] = \sum_{q=1}^{m_2} b_{ij}^q \, \frac{\partial}{\partial t_q} \, , \qquad 1 \le i < j \le m \, .
$$
For later purposes we also adopt the following notation from \cite{gt:h-type}:
$$
\beta_i^q(\boldx) := - \sum_{j=1}^m b_{ij}^q x_j, \qquad i=1,\ldots,m , q = 1,\ldots,m_2 \,.
$$

\subsection{Calculus in step two Carnot groups with taming Riemannian metric.}

In this subsection we assume that $\G$ is equipped with a taming Riemannian metric $g = g_h \oplus g_v$, and that the basis $X_1,\ldots,X_m$ and $T_1,\ldots,T_{m_2}$ are $g_h$-orthonormal and $g_v$-orthonormal respectively. The formulas presented in this subsection are well known. See, for example, \cite{gt:h-type}.

Let us denote by $b := ||\boldx||_h^2$ and $c_v := ||\boldt||_v^2$, so that the function $a_v$ in Definition \ref{def:Kaplan-norm} satisfies
$$
a_v = b^2 + 16 c_v \, .
$$
(We continue to highlight expressions which depend on the choice of the vertical metric $g_v$.) We also introduce, for a sufficiently smooth real-valued function $u$ on $\G$, the notation
$$
u_{,ij} := \tfrac12 (X_i X_j u + X_j X_i u), \qquad i,j = 1,\ldots,m \, .
$$
The $m\times m$ matrix $(u_{,ij})$ represents the {\bf symmetrized horizontal Hessian} $(D^2_H)^* u$ of $u$. 

\begin{lemma}\label{lem:0}
For each $i,j=1,\ldots,m$, $X_i(x_j) = \delta_{ij}$. Hence $\nabla_0 b = 2\boldx$.
\end{lemma}

\begin{lemma}\label{lem:1}
For each $i,j,k=1,\ldots,m$, $X_iX_j(x_k) = 0$. Hence $X_iX_j(b) = 2\delta_{ij}$ and $b_{,ij} = 2\delta_{ij}$, so $(D^2_H)^* b = 2\Id$.
\end{lemma}

\begin{lemma}\label{lem:2}
For each $i=1,\ldots,m$ and $q=1,\ldots,m_2$, $X_i(t_q) = \tfrac12 \beta_i^q$. Hence $\nabla_0 c_v =  \boldJ_{\boldt}(\boldx)$.
\end{lemma}

We pause here to clarify that the notation $\boldJ_{\boldt}(\boldx)$ denotes the function $\ttg \mapsto \boldJ_{\boldt(\ttg)}(\boldx(\ttg)$ mapping $\G$ into $\fv_1$. For the reader's convenience we also sketch the proof of the above formula for $\nabla_0 c_v$. It suffices to show that $X_i(c_v) = \langle \boldJ_\boldt(\boldx),X_i\rangle_h$ for all $i=1,\ldots,m$. To this end, we compute
\begin{equation*}\begin{split}
X_i(c_v) &= \sum_{q=1}^{m_2} 2 t_q X_i(t_q) = \sum_q t_q \beta_i^q \\
&= - \sum_j \sum_q \beta_{ij}^q x_j t_q = - \sum_j \sum_q x_j t_q \langle [X_i,X_j],T_q \rangle_v \\
&= \langle \boldJ_{\sum_q t_q T_q}(\sum_j x_j X_j), X_i \rangle_h = \langle \boldJ_{\boldt}(\boldx), X_i \rangle_h.
\end{split}\end{equation*}

\begin{lemma}\label{lem:3}
For each $i,j=1,\ldots,m$ and $q=1,\ldots,m_2$, $X_iX_j(t_q) = \tfrac12 b_{ij}^q$. Hence $X_iX_j(c_v) = \langle \boldJ_{\boldt}(X_i),X_j \rangle_h + \tfrac12 \sum_q \beta_i^q\beta_j^q$ and $(c_v)_{,ij} = \tfrac12 \sum_q \beta_i^q\beta_j^q$.
\end{lemma}

\begin{lemma}\label{lem:4}
For each $i=1,\ldots,m$, $X_i(a_v) = \langle 4 ||\boldx||_h^2 \boldx + 16 \boldJ_{\boldt}(\boldx), X_i\rangle_h$. Hence 
$$
||\nabla_0(a_v)||_h^2 = 16( ||\boldx||_h^6 + 16||\boldJ_{\boldt}(\boldx)||_h^2).
$$
Moreover, $X_i(N_v) = N_v^{-3} \langle ||\boldx||_h^2 \boldx + 4 \boldJ_{\boldt}{\boldx}, X_i\rangle_h$ and hence
$$
||\nabla_0 N_v||_h^2 = N_v^{-6} (||\boldx||_h^6 + 16 ||\boldJ_{\boldt}{\boldx}||_h^2 ) \, .
$$
Here $N_v$ denotes Kaplan's quasinorm associated to the given vertical metric $g_v$.
\end{lemma}

\subsection{Characterizing H-type groups via the horizontal derivative of Kaplan's quasinorm}

Considering the conclusion in Lemma \ref{lem:4} and using \eqref{eq:h-type-characterization-2}, we conclude the well-known fact that if $\G$ is H-type, then 
$$
||\nabla_0 N||_h^2 = \frac{||\boldx||_h^2 (||\boldx||_h^4 + 16 ||\boldt||_v^2)}{N^6} = \frac{||\boldx||_h^2}{N^2}.
$$
See also \eqref{eq:norm-horiz-grad}. In this subsection, we quantify this identity within the broader class of step two Carnot groups. More precisely, we estimate the deviation of $||\nabla_0 N_v||_h^2$ from its value in the H-type case. 

\begin{proposition}\label{prop:H-type-stability-1}
Let $\G$ be a step two Carnot group, with metric $g_h$ defined in the horizontal layer $\fv_1$ of $\fg$. Then
\begin{equation}\label{eq:H-type-stability-1}
\inf_{g_v} \sup_{0 \ne \ttg \in \G} \left| \frac{||\boldx(\ttg)||_h^2}{N_v(\ttg)^2} - ||(\nabla_0 N_v)(\ttg)||_h^2 \right| \le \sqrt{m} \, \delta(\G).
\end{equation}
\end{proposition}

\begin{proof}
Fix a vertical metric $g_v$ and fix $\ttg = \exp(\boldx+\boldt) \in \G$, $\ttg \ne 0$. From Lemma \ref{lem:4} we conclude (supressing the argument $\ttg \in \G$ for convenience) that
\begin{equation*}\begin{split}
\frac{||\boldx||_h^2}{N_v^2} - ||\nabla_0 N_v||_h^2
& = \frac{||\boldx||_h^2}{N_v^2} - \frac{||\boldx||_h^6 + 16 ||\boldJ_{\boldt}(\boldx)||_h^2}{N_v^6} \\
& = 16 \, \frac{||\boldx||_h^2 \, ||\boldt||_v^2 - ||J_{\boldt}(\boldx)||_h^2}{N_v^6} \\
& = 16 \, \frac{||\boldx||_h^2 \, ||\boldt||_v^2 + \langle \boldJ_{\boldt}^2(\boldx),\boldx \rangle_h}{N_v^6} \\
& = 16 \, \frac{\langle (\boldJ_{\boldt}^2 + ||\boldt||_v^2 \, \Id)(\boldx),\boldx \rangle_h}{N_v^6} \, .
\end{split}\end{equation*}
We conclude that
\begin{equation*}\begin{split}
\sup_{0 \ne \ttg \in \G} \left| \frac{||\boldx(\ttg)||_h^2}{N_v(\ttg)^2} - ||(\nabla_0 N_v)(\ttg)||_h^2 \right|
& \le 16 \, \frac{||\boldx||_h^2}{N_v^6} \, || \boldJ_{\boldt}^2 + ||\boldt||_v^2 \, \Id ||_{op} \\
& \le 16 \,  \frac{||\boldx||_h^2}{N_v^6} \, || \boldJ_{\boldt}^2 + ||\boldt||_v^2 \, \Id ||_{HS} \\
& \le \sqrt{m} \, \frac{(16||\boldt||_v^2) \, ||\boldx||_h^2}{N_v^6} \, \delta_{0}(\G,g_v) \\
& \le \sqrt{m} \, \delta_0(\G,g_v).
\end{split}\end{equation*}
Taking the infimum over all vertical metrics completes the proof.
\end{proof}

In fact, the quantity on the left hand side of \eqref{eq:H-type-stability-1} is comparable to $\delta(\G)$.

\begin{proposition}\label{prop:H-type-stability-1a}
Let $\G$ be a step two Carnot group, with metric $g_h$ defined in the horizontal layer $\fv_1$ of $\fg$ and $\dim \fv_1 = m$. Then
$$
\delta(\G) \le C \, \inf_{g_v} \sup_{0 \ne \ttg \in \G} \left| \frac{||\boldx(\ttg)||_h^2}{N_v(\ttg)^2} - ||(\nabla_0 N_v)(\ttg)||_h^2 \right\| \, .
$$
Here $C$ denotes an absolute positive constant.
\end{proposition}

\begin{proof}[Proof of Proposition \ref{prop:H-type-stability-1a}]
It suffices to prove, for each fixed metric $g_v$ on the vertical layer $\fv_2$, that
\begin{equation}\label{eq:H-type-stability-1a-equation}
\delta(\G,g_v) \le C \sup_{0 \ne \ttg=\exp(\boldx+\boldt) \in \G} \left| \frac{||\boldx(\ttg)||_h^2}{N_v(\ttg)^2} - ||(\nabla_0 N_v)(\ttg)||_h^2 \right| \, .
\end{equation}
Noting that
$$
\ttg \mapsto \frac{||\boldx(\ttg)||_h^2}{N_v(\ttg)^2} - ||(\nabla_0 N_v)(\ttg)||_h^2
$$
is $0$-homogeneous, we conclude that it is equivalent to prove the estimate
$$
\delta(\G,g_v) \le C \sup_{\substack{\ttg=\exp(\boldx+\boldt) \\ ||\boldt||_v = 1}} \left| \frac{||\boldx(\ttg)||_h^2}{N_v(\ttg)^2} - ||(\nabla_0 N_v)(\ttg)||_h^2 \right| \,.
$$
The proof of Proposition \ref{prop:H-type-stability-1} shows that
\begin{equation}\label{eq:H-type-stability-1a-1}
\left| \frac{||\boldx||_h^2}{N_v^2} - ||\nabla_0 N_v||_h^2 \right| = 16 \left| \frac{\langle (\boldJ_\boldt^2 + \Id)(\boldx),\boldx \rangle_h}{N_v^6} \right|
\end{equation}
for any $\ttg = \exp(\boldx+\boldt) \in \G$ with $||\boldt||_v = 1$.

For fixed $\boldt \in \fv_2$ with $||\boldt||_v = 1$ and for fixed $L>0$, let $\boldx = \boldx_\boldt$ be an element of $\fv_1$ with $||\boldx||_h = L$, so that $\boldx$ is an eigenvector of $\boldJ_\boldt^2 + \Id$ with maximal eigenvalue. We will choose a suitable value for $L$ momentarily. We observe that $|\langle (\boldJ_\boldt^2 + \Id)(\boldx),\boldx \rangle_h| = ||\boldJ_\boldt^2 + \Id||_{op} \cdot ||\boldx_\boldt||_h^2$. For such a choice of $\boldx$ and $\boldt$, we rewrite \eqref{eq:H-type-stability-1a-1} as
$$
\left| \frac{||\boldx_\boldt||_h^2}{N_v^2} - ||\nabla_0 N_v||_h^2 \right| = 16 \frac{||\boldx_\boldt||_h^2}{N_v^6} ||\boldJ_\boldt^2 + \Id||_{op}.
$$
Hence
\begin{equation*}\begin{split}
\delta(\G,g_v) 
= \frac1{\sqrt{m}} \sup_{\substack{\boldt \\ ||\boldt||_v=1}} ||\boldJ_\boldt^2+\Id||_{HS}
&\le \sup_{\substack{\boldt \\ ||\boldt||_v=1}} ||\boldJ_\boldt^2+\Id||_{op} \\
&= \sup_{\substack{ \boldt \\ ||\boldt||_v=1}} \frac{N_v(\boldx_\boldt,\boldt)^6}{16||\boldx_t||_h^2} \, \left| \frac{||\boldx_t||_h^2}{N_v(\boldx_\boldt,\boldt)^2} - ||(\nabla_0 N_v)(\boldx_\boldt,\boldt)||_h^2 \right| \,.
\end{split}\end{equation*}
Since $||\boldx_\boldt||_h = L$ and $N_v(\boldx,\boldt) = (||\boldx||_h^4 + 16||\boldt||_v^2)^{1/4}$, we obtain
$$
\delta(\G,g_v)  \le \frac{(L^4+16)^{3/2}}{16L^2} \, \sup_{\substack{ \boldt \\ ||\boldt||_v=1}}\left| \frac{||\boldx_\boldt||_h^2}{N_v(\boldx_\boldt,\boldt)^2} - ||(\nabla_0 N_v)(\boldx_\boldt,\boldt)||_h^2 \right| \,.
$$
We now choose the value of $L>0$ to minimize the expression on the right hand side. The function $h(L) = (L^4+16)^{3/2}/16L^2$ achieves its minimum over $0<L<\infty$ when $L=2^{3/4}$, and $h(2^{3/4}) = 3\sqrt{3}/2 =: C$. Hence
\begin{equation*}\begin{split}
\delta(\G,g_v) 
\le C \sup_{\substack{\boldt \\ ||\boldt||_v =1 }} \left| \frac{||\boldx_\boldt||_h^2}{N_v(\boldx_\boldt,\boldt)^2} - ||(\nabla_0 N_v)(\boldx_\boldt,\boldt)||_h^2 \right| 
\le C \sup_{\boldx,\boldt} \left| \frac{||\boldx||_h^2}{N_v(\boldx,\boldt)^2} - ||(\nabla_0 N_v)(\boldx,\boldt)||_h^2 \right| \\
\end{split}\end{equation*}
which completes the proof of \eqref{eq:H-type-stability-1a-equation}.
\end{proof}

Theorem \ref{thm:main-theorem-1} is an immediate corollary of the preceding proposition.

\begin{proof}[Proof of Theorem \ref{thm:main-theorem-1}]
Let $\G$ be a step two Carnot group satisfying the assumptions of Theorem \ref{thm:main-theorem-1} for some vertical metric $g_v$. Proposition \ref{prop:H-type-stability-1a} implies that $\delta(\G) = 0$. By Proposition \ref{prop:delta_0-characterization-of-nascent-H-type-groups}, the group $\G$ equipped with the taming Riemannian metric $g_h \oplus g_v$ is an H-type group. 
\end{proof}

\subsection{Characterizing H-type groups via the sub-Laplacian of Kaplan's quasinorm}

We now derive similar two-sided stability estimates for 
$$
\cL(N_v^{2-Q}),
$$
where $N_v$ denotes Kaplan's quasinorm. Theorem \ref{thm:main-theorem-2} follows easily from these estimates.

\begin{proposition}\label{prop:H-type-stability-2}
Let $\G$ be a step two Carnot group, with metric $g_h$ defined in the horizontal layer $\fv_1$ of $\fg$ and $\dim \fv_1 = m$. Then
$$
\frac1C \, \delta(\G) \le \inf_{g_v} \sup_{0 \ne \ttg \in \G} \left| N_v(\ttg) \, (\cL N_v)(\ttg) - (Q-1) ||(\nabla_0 N_v)(\ttg)||_h^2 \right| \le C \, \delta(\G)
$$
for some constant $C = C(m,m_2)$ depending only on the dimensions $m = \dim \fv_1$ and $m_2 = \dim \fv_2$.
\end{proposition}

As before, we note that the expression $N_v \, \cL N_v - (Q-1) ||\nabla_0 N_v||_h^2$ is $0$-homogeneous.

\begin{proof}[Proof of Proposition \ref{prop:H-type-stability-2}]
We again fix a vertical metric $g_v$. First, we show that
\begin{equation}\label{eq:H-type-stability-2a}
\sup_{0 \ne \ttg \in \G} \left| N_v(\ttg) \, (\cL N_v)(\ttg) - (Q-1) ||(\nabla_0 N_v)(\ttg)||_h^2 \right| \le C \, \delta(\G,g_v) \, .
\end{equation}
The following identities hold for all points $\ttg \in \G$, $\ttg \ne 0$, but we omit the argument $\ttg$ for ease of exposition. From Lemma \ref{lem:0} we conclude that $\cL(b) = 2m$ and hence
$$
\cL(b^2) = 2 b \, \cL(b) + 2 ||\nabla_0 b||_h^2 = 4m \, b + 8 \, b = 4(m+2)b.
$$
Moreover,
$$
\cL(c_v) = \frac12 \sum_{i=1}^m \sum_{q=1}^{m_2} (\beta_i^q)^2 = \frac12 \sum_{q=1}^{m_2} ||\boldJ_{T_q}(\boldx)||_h^2
$$
for some choice of a $g_v$-orthonormal basis $\varepsilon_1,\ldots,\varepsilon_{m_2}$ for $\fv_2$. Thus
\begin{equation*}\begin{split}
\cL a_v = \cL(b^2 + 16 c_v) & = 4(m+2) b + 8 \sum_{q=1}^{m_2} ||\boldJ_{\varepsilon_q}(\boldx)||_h^2 \\
& = 4(m+2) b - 8 \sum_{q=1}^{m_2} \langle \boldJ_{\varepsilon_q}^2(\boldx),\boldx \rangle_h \\
& = 4(Q+2) b - 8 \sum_{q=1}^{m_2} \langle (\boldJ_{\varepsilon_q}^2 + \Id) (\boldx),\boldx \rangle_h \\
\end{split}\end{equation*}
where we used the facts that $b = ||\boldx||_h^2$ and $Q = m+2m_2$.

Next, since $a_v = N_v^4$, we have
$$
\cL a_v = 4 N_v^3 \, \cL N_v + 12 N_v^2 \, ||\nabla_0 N_v||_h^2.
$$
Using Lemma \ref{lem:4} and the preceding expression for $\cL a_v$ we conclude that
\begin{equation}\label{eq:NLN-identity}
N_v \, \cL N_v - (Q-1) \, ||\nabla_0 N_v||_h^2 = 16 (Q+2) \frac{\langle (\boldJ_{\boldt}^2 + ||\boldt||_v^2 \, \Id)\boldx,\boldx \rangle_h}{N_v^6} - 2 \frac{\sum_{q=1}^{m_2} \langle (\boldJ_{\varepsilon_q}^2 + \Id)\boldx,\boldx \rangle_h}{N_v^2} \, .
\end{equation}
Thus
\begin{equation*}\begin{split}
\Bigl| N_v \, \cL N_v - (Q-1) \, ||\nabla_0 N_v||_h^2 \Bigr| 
&\le 16 (Q+2) \sqrt{m} \delta(\G,g_v) \, \frac{||\boldt||_v^2 \, ||\boldx||_h^2}{N_v^6} + 2m_2 \sqrt{m} \delta(\G,g_v) \, \frac{||\boldx||_h^2}{N_v^2} \\
&\le C(m,m_2) \delta(\G,g_v)
\end{split}\end{equation*}
since $||\boldx||_h \le \boldN_v$ and $4||\boldt||_v \le \boldN_v^2$. This completes the proof of \eqref{eq:H-type-stability-2a}.

Next, we will show that
\begin{equation}\label{eq:H-type-stability-2b}
\delta(\G,g_v) \le C \, \sup_{0 \ne \ttg \in \G} \left| N_v(\ttg) \, (\cL N_v)(\ttg) - (Q-1) ||(\nabla_0 N_v)(\ttg)||_h^2 \right| \, .
\end{equation}
To this end, we return to \eqref{eq:NLN-identity}, which we consider for points $\ttg = \exp(\boldx + \boldt)$ with $||\boldt||_v = 1$:
\begin{equation*}
(N_v \, \cL N_v - (Q-1) \, ||\nabla_0 N_v||_h^2)(\boldx,\boldt) = 16 (Q+2) \frac{\langle (\boldJ_{\boldt}^2 + \Id)\boldx,\boldx \rangle_h}{N_v(\boldx,\boldt)^6} - 2 \frac{\sum_{q=1}^{m_2} \langle (\boldJ_{\varepsilon_q}^2 + \Id)\boldx,\boldx \rangle_h}{N_v(\boldx,\boldt)^2} \, .
\end{equation*}
We select $\boldt \in \fv_2$ with $||\boldt||_v = 1$ so that
$$
||\boldJ_\boldt^2 + \Id||_{op} = \sup_{\substack{T \in \fv_2 \\ ||T||_v = 1}} ||\boldJ_T^2 + \Id||_{op}
$$
As before, we fix $L>0$ (the choice of $L$ will be determined momentarily) and we let $\boldx = \boldx_\boldt$ be an element of $\fv_1$ with $||\boldx||_h = L$, so that $\boldx$ is an eigenvector of $\boldJ_\boldt^2 + \Id$ with maximal eigenvalue. Then $|\langle (\boldJ_\boldt^2 + \Id)(\boldx),\boldx \rangle_h| = ||\boldJ_\boldt^2 + \Id||_{op} \cdot ||\boldx_\boldt||_h^2$ and we obtain
\begin{equation*}\begin{split}
(N_v \, \cL N_v - 
&(Q-1) \, ||\nabla_0 N_v||_h^2)(\boldx_\boldt,\boldt) \\
&\ge 16 (Q+2) \frac{||\boldJ_{\boldt}^2 + \Id||_{op} ||\boldx||_h^2}{N_v(\boldx,\boldt)^6} - 2 \frac{\sum_{q=1}^{m_2} |\langle (\boldJ_{\varepsilon_q}^2 + \Id)\boldx,\boldx \rangle_h|}{N_v(\boldx,\boldt)^2} \\
&\ge 16 (Q+2) \frac{||\boldJ_{\boldt}^2 + \Id||_{op} ||\boldx||_h^2}{N_v(\boldx,\boldt)^6} - 2m_2 \left( \sup_{\substack{T \in \fv_2 \\ ||T||_v = 1}} ||\boldJ_T^2 + \Id||_{op} \right)
 \frac{||\boldx_\boldt||_h^2}{N_v(\boldx,\boldt)^2} \\
&= \left( \frac{16 (Q+2) L^2}{(L^4+16)^{3/2}} - \frac{2m_2 L^2}{\sqrt{L^4+16}} \right) \sup_{\substack{T \in \fv_2 \\ ||T||_v = 1}} ||\boldJ_T^2 + \Id||_{op} \\
&\ge \frac1{\sqrt{m}} \left( \frac{16 (Q+2) L^2}{(L^4+16)^{3/2}} - \frac{2m_2 L^2}{\sqrt{L^4+16}} \right) \sup_{\substack{T \in \fv_2 \\ ||T||_v = 1}} ||\boldJ_T^2 + \Id||_{HS} \, .
\end{split}\end{equation*}
The expression
$$
h(L) = \frac{16 (Q+2) L^2}{(L^4+16)^{3/2}} - \frac{2m_2 L^2}{\sqrt{L^4+16}}
$$
is maximized over $0<L<\infty$ when
$$
L = L_0 := \frac{2^{3/4}(m+2)^{1/4}}{(Q+2+m_2)^{1/4}},
$$
and
$$
h(L_0) = \frac{2(m+2)^{3/2}}{3\sqrt{3}\sqrt{Q+2}}.
$$
(Recall that $Q = m + 2m_2$.) Hence
\begin{equation*}\begin{split}
\delta(\G,g_v) 
&= \frac1{\sqrt{m}} \, \sup_{\substack{T \in \fv_2 \\ ||T||_v = 1}} ||\boldJ_T^2 + \Id||_{HS} \\
&\le  C \, \sup_{\substack{\boldt \\ ||\boldt||_v = 1}}  \left| (N_v \, \cL N_v - (Q-1) \, ||\nabla_0 N_v||_h^2)(\boldx_\boldt,\boldt) \right| \\
&\le C \, \sup_{\substack{\boldx,\boldt}}  \left| (N_v \, \cL N_v - (Q-1) \, ||\nabla_0 N_v||_h^2)(\boldx,\boldt) \right| \\
\end{split}\end{equation*}
with
$$
C = C(m,m_2) = \frac{3\sqrt{3(Q+2)}}{2(m+2)^{3/2}}.
$$
This completes the proof of \eqref{eq:H-type-stability-2b}.
\end{proof}

Using the identity
$$
\cL u_v(\ttg) = (2-Q) N_v(\ttg)^{1-Q} (\cL N_v)(\ttg) + (2-Q)(1-Q) N_v(\ttg)^{-Q} ||(\nabla_0 N_v(\ttg)||_h^2,
$$
valid at all points $\ttg \in \G$, $\ttg \ne 0$, we deduce the following corollary.

\begin{corollary}\label{cor:H-type-stability-1b}
Under the assumptions of the previous theorem, with $u_v := N_v^{2-Q}$, we have
$$
\frac1C \delta(\G) \le C \inf_{g_v} \sup_{0 \ne \ttg \in \G} N_v(\ttg)^Q | (\cL u_v)(\ttg)| \le C \delta(\G) \, .
$$
\end{corollary}

\begin{proof}[Proof of Theorem \ref{thm:main-theorem-2}]
Let $\G$ be a step two Carnot group satisfying the assumptions of Theorem \ref{thm:main-theorem-2} for some vertical metric $g_v$. Corollary \ref{cor:H-type-stability-1b} implies that $\delta_M(\G) = 0$. By Proposition \ref{prop:delta_0-characterization-of-nascent-H-type-groups}, $\G$ equipped with the metric $g_h \oplus g_v$ is of H-type. 
\end{proof}

\section{Polarizability}\label{sec:polarizability}

Polarizable Carnot groups were introduced in \cite{bt:polar}. This class of Carnot groups admits a coherent family of singular solutions for all of the $p$-Laplacian operators, and also admits a polar coordinate integration formula for which the radial curves are horizontal. See the introduction for further information.

\begin{definition}[Balogh--Tyson]\label{def:polarizable}
Let $\G$ be a Carnot group of homogeneous dimension $Q \ge 3$, and let $u$ be Folland's fundamental solution for the sub-Laplacian in $\G$. Then $\G$ is {\bf polarizable} if the function $u^{1/(2-Q)}$ is $\infty$-harmonic in the complement of $0 \in \G$.
\end{definition}

In \cite{bt:polar} we showed that all H-type groups are polarizable, and explicitly computed the horizontal polar coordinate integration formula in that setting. We conjectured that H-type groups are the only examples of polariable groups. Here we restate that conjecture in the language introduced in this paper.

\begin{conjecture}\label{conj:full-polarizability-conjecture}
Let $\G$ be a polarizable Carnot group. Then $\G$ has step two, and $\G$ is a nascent H-type group.
\end{conjecture}

In this section we discuss a possible approach to Conjecture \ref{conj:full-polarizability-conjecture} for step two Carnot groups via the H-type deviation. We restate Conjecture \ref{conj:polarizability-conjecture} from the introduction.

\begin{conjecture}\label{conj:polarizability-conjecture-2}
Let $\G$ be a step two Carnot group with metric $g_h$ defined in the first layer of the Lie algebra. Denote by $\Sigma = \{\ttg = \exp(\boldx + \boldt) \in \G : \boldt = 0 \}$. Let $u$ be the fundamental solution for the sub-Laplacian $\cL$ and let $N = u^{1/(2-Q)}$. Then
\begin{equation}\label{eq:polarizability-conjecture-2}
\sup_{\ttg \in \Sigma} |N(\ttg)(\cL_\infty N)(\ttg)| \simeq \delta(\G)^2 \, .
\end{equation}
\end{conjecture}

If Conjecture \ref{conj:polarizability-conjecture-2} holds true, then Conjecture \ref{conj:full-polarizability-conjecture} holds for step two groups. Indeed, if $\G$ is polarizable then $\cL_\infty N$ vanishes on all of $\G$; \eqref{eq:polarizability-conjecture-2} then implies that $\delta(\G) = 0$, after which Proposition \ref{prop:delta_0-characterization-of-nascent-H-type-groups} implies that $\G$ is a nascent H-type group.

The main result of this section is the following theorem.

\begin{theorem}\label{th:anisotropic-Heis-groups}
Conjecture \ref{conj:polarizability-conjecture-2} holds true for the anisotropic Heisenberg groups $\Heis^n(\tfrac12,1,\ldots,1)$.
\end{theorem}

The group $\Heis^n(\tfrac12,1,\ldots,1)$ is modeled by the underlying space $\R^{2n+1}$ with horizontal distribution defined at each point $\ttg = (x,y,t)$ as the span of the following vector fields:
$$
X_1 = \deriv{x_1} - \frac12 y_1 \deriv{t}, \qquad Y_1 = \deriv{y_1} + \frac12 x_1 \deriv{t},
$$
$$
X_j = \deriv{x_j} - y_j \deriv{t}, \quad Y_j = \deriv{y_j} + \frac12 x_j \deriv{t}, \qquad \quad 2 \le j \le n.
$$
The vertical direction is spanned at each point by $T = \deriv{t}$. We observe the nontrivial bracket relations $[X_1,Y_1] = T$ and $[X_j,Y_j] = 2T$ for all $2 \le j \le n$. To simplify later formulas we also identify the group with $\C^n \times \R$ by writing $\ttg = (x,y,t) = (z,t)$ with $z = x + \bi y$.

Denote by $\nabla_0 u = (X_1u,Y_1u,X_2u,Y_2u,\ldots,X_nu,Y_nu)$ the horizontal gradient of a function $u:\Heis^n(\tfrac12,1,\ldots,1) \to \R$, by $\diver_0(\cX) = X_1(a_1)+Y_1(b_1)+\cdots+X_n(a_n)+Y_n(b_n)$ the horizontal divergence of a horizontal vector field $\cX = a_1 X_1 + b_1 Y_1 + \cdots + a_n X_n + b_n Y_n$, and by
$$
\cL u = \diver_0(\nabla_0 u) = \sum_{j=1}^n X_j^2 u + Y_j^2 u
$$
and
$$
\cL_\infty u = \frac12 \langle \nabla_0 (||\nabla_0 u||_h^2), \nabla_0 u \rangle_h
$$
the Laplacian and $\infty$-Laplacian, respectively. 

An explicit computation of the Beals--Gaveau--Greiner formula \eqref{eq:bgg-formula} can be carried out in this case. The result is summarized in the following theorem. See \cite{bt:polar} for the case $n=2$ and \cite{bdz:anisotropic} for the case of general $n \ge 2$. Note that we have adjusted the definitions of the intermediate variables $A$, $B$, and $C$ from those references, to aid in subsequent computations. We also remark that fundamental solutions for sub-Laplacians (and a variety of other differential operators) on anisotropic Heisenberg groups were considered by Chang and Tie in \cite{ct:anisotropic1}, \cite{ct:anisotropic2}.

\begin{theorem}
In $\Heis^n(\tfrac12,1,\ldots,1)$, with $z_j = (x_j,y_j)$, $1\le j\le n$, and $z'=(z_2,\ldots,z_n)$, introduce the following expressions:
\begin{equation}\label{eq:A}
A := \frac14 |z_1|^2,
\end{equation}
\begin{equation}\label{eq:B}
B := \frac14 |z_1|^2 + \frac12 ||z'||^2,
\end{equation}
and
\begin{equation}\label{eq:C}
C := \sqrt{B^2 + t^2}.
\end{equation}
Then the fundamental solution for the sub-Laplacian $\cL$ is given by a constant multiple of
\begin{equation}\label{eq:anisotropic-u}
u(z,t) = \frac{(B+C)^{1/2}}{C (A+C)^{n-1/2}} \, .
\end{equation}
\end{theorem}

The corresponding homogeneous norm is
\begin{equation}\label{eq:anisotropic-N}
N(z,t) = 2^{1/2 + 1/4n} \frac{C^{1/2n}(A+C)^{1/2-1/4n}}{(B+C)^{1/{4n}}}.
\end{equation}

\begin{remark}
The coefficient $2^{1/2+1/(4n)}$ is included in this expression is to simplify the description of the intersection of the unit $N$-sphere $S_N$ with the horizontal plane $\Sigma = \{ t = 0 \}$. Indeed, note that if $(z,t) \in \Sigma$, then $B=C$ and
$$
N(z,0) = \sqrt{2} B^{\tfrac1{4n}} (A+B)^{\tfrac12-\tfrac1{4n}} \, .
$$
Hence $(z,0) \in S_N \cap \Sigma$ if and only if
$$
4^n B (A+B)^{2n-1} = 1,
$$
that is,
$$
(\frac12 |z_1|^2 + ||z'||^2)(|z_1|^2 + ||z'||^2)^{2n-1} = 1.
$$
In particular, a point $(0,z',0)$ lies in $S_N \cap \Sigma$ if and only if $||z'|| = 1$.
\end{remark}

In order to prove Theorem \ref{th:anisotropic-Heis-groups} we need to know the value of the H-type deviation. This was computed in Example \ref{ex:anisotropic-Heisenberg}. With $b_1 = \tfrac12$ and $b_2=\cdots=b_n=1$ we obtain
$$
\delta(\Heis^n(\tfrac12,1,\ldots,1)) = \sqrt{1-\frac{(n^{-1}(n-1+\tfrac14))^2}{n^{-1}(n-1+\tfrac1{16})}} = \sqrt{\frac{9(n-1)}{n(16n-15)}} \simeq \frac1{\sqrt{n}},
$$
where the notation $a(n) \simeq b(n)$ means that $0<\liminf a(n)/b(n) \le \limsup a(n)/b(n) < \infty$. It thus suffices for us to prove that
$$
\sup_{\ttg \in \Sigma} |N(\ttg)(\cL_\infty N)(\ttg)|  \simeq \frac1n.
$$
Noting as before that $N \, \cL_\infty N$ is $0$-homogeneous, we conclude that it is equivalent for us to prove that
\begin{equation}\label{eq:NLN-asymptotics}
\sup_{\substack{\ttg \in \Sigma \\ N(\ttg) = 1}} |N(\ttg)(\cL_\infty N)(\ttg)|  \simeq \frac1n.
\end{equation}
We emphasize here that the implicit comparison constants in \eqref{eq:NLN-asymptotics} do {\it not} depend on the dimension $n$ associated to the anisotropic Heisenberg group.

\smallskip

\noindent It will be more convenient to calculate using the logarithms of $N$ and $u$. To this end, we use the following proposition relating the $\infty$-Laplacians of a positive function $v$ and of its logarithm.

\begin{proposition}\label{prop:log}
Let $v$ be a positive function defined on a Carnot group $\G$. Then
$$
\nabla_0(\log v) = \frac{\nabla_0 v}{v},
$$
$$
||\nabla_0(\log v)||_h^2 = \frac{||\nabla_0 v||_h^2}{v^2},
$$
\begin{equation}\label{eq:Laplacian-log-v}
\cL(\log v) = \frac{\cL v}{v} - \frac{||\nabla_0 v||_h^2}{v^2},
\end{equation}
and
\begin{equation}\label{eq:infinity-Laplacian-log-v}
\cL_\infty(\log v) = \frac{\cL_\infty v}{v^3} - \frac{||\nabla_0 v||_h^4}{v^4}.
\end{equation}
\end{proposition}

The proof is elementary and will be omitted.

In the following lemma, we record the first and second derivatives of the expressions $A$, $B$, and $C$ defined in \eqref{eq:A}, \eqref{eq:B}, and \eqref{eq:C}.

\begin{lemma}\label{lem:derivatives-ABCt}
Let $\G = \Heis^n(\tfrac12,1,\ldots,1)$. For the functions $A = \tfrac14|z_1|^2$,  $B = \tfrac14|z_1|^2 + \tfrac12||z'||^2$, $t$, and $C = \sqrt{B^2+t^2}$ on $\G$, we have
$$
\nabla_0 A = \begin{pmatrix} \tfrac12 z_1 \\ 0 \end{pmatrix} =: \vec{u},
$$
$$
\nabla_0 B = \begin{pmatrix} \tfrac12 z_1 \\ z' \end{pmatrix} =: \vec{v},
$$
$$
\nabla_0 t =  \begin{pmatrix} \tfrac12 \bi z_1 \\ \bi z' \end{pmatrix} =: \vec{w},
$$
and
$$
\nabla_0 C = \frac{B}{C} \vec{v} + \frac{t}{C} \vec{w}.
$$
\end{lemma}

We also note that
\begin{equation}\label{eq:uvw-1}
||\vec{v}||^2 = ||\vec{w}||^2 = 2B-A,
\end{equation}
\begin{equation}\label{eq:uvw-2}
||\vec{u}||^2 = \langle \vec{u},\vec{v} \rangle = A,
\end{equation}
and
\begin{equation}\label{eq:uvw-3}
\langle \vec{v},\vec{w} \rangle = \langle \vec{u},\vec{w} \rangle = 0.
\end{equation}
The inner product in \eqref{eq:uvw-2} and \eqref{eq:uvw-3} between elements $\vec{a}=(a_1,\ldots,a_n)^T$ and $\vec{b}=(b_1,\ldots,b_n)^T$ in $\C^n$ is $\langle \vec{a},\vec{b} \rangle = \Real(\sum_{j=1}^n a_j\overline{b_j})$, and the norm in \eqref{eq:uvw-1} and \eqref{eq:uvw-2} is $||\vec{a}||^2 = \langle \vec{a},\vec{a} \rangle$.

Finally, we introduce the notation
$$
\bar{n} := n - \tfrac12
$$
to simplify upcoming formulas.

\begin{proposition}\label{prop:grad-log-u}
For the function $u$ in \eqref{eq:anisotropic-u}, we have
\begin{equation}\label{eq:grad-log-u-1}
\frac{\nabla_0 u}{u} = (P - \bar{n} \frac{B}{C} R) \vec{v} - (Q + \bar{n} \frac{t}{C} R) \vec{w} - \bar{n} R \vec{u},
\end{equation}
where
\begin{equation}\label{eq:PQR}
P = \frac{C-2B}{2C^2}, \qquad Q = \frac{(C+2B)t}{2C^2(C+B)}, \qquad \mbox{and} \qquad R = \frac1{C+A}.
\end{equation}
Furthermore,
\begin{equation}\label{eq:grad-log-u-2}
\frac{||\nabla_0 u||_h^2}{u^2} = \frac{2B-A}{2C(C+B)} + \bar{n} \frac{2(AB-AC+BC)}{C^2(C+A)} + \bar{n}^2\frac{2B}{C(C+A)} \,.
\end{equation}
\end{proposition}

\begin{proof}
Since $\log(u) = \tfrac12 \log(B+C) - \log(C) - \bar{n} \log(A+C)$ we have
$$
\frac{\nabla_0 u}{u}  = \frac12 \frac{\nabla_0(B+C)}{B+C} - \frac{\nabla_0 C}{C} - \bar{n} \frac{\nabla_0(A+C)}{A+C}\,.
$$
Formula \eqref{eq:grad-log-u-1} now follows by using the formulas for the horizontal gradients of $A$, $B$ and $C$ in the previous lemma. To obtain \eqref{eq:grad-log-u-2}, we use the identity $t^2 = C^2 - B^2$ to conclude that
$$
P^2 + Q^2 = \frac{(C-2B)^2}{4C^4} + \frac{(C+2B)^2t^2}{4C^4(C+B)^2} = \frac{1}{2C(C+B)}
$$
and
$$
tQ - BP = \frac{(C+2B)t^2}{2C^2(C+B)} - \frac{(C-2B)B}{2C^2} = \frac12.
$$
The conclusion follows from \eqref{eq:grad-log-u-1} by algebraic manipulations and \eqref{eq:uvw-1}, \eqref{eq:uvw-2}, and \eqref{eq:uvw-3}.
\end{proof}

Since $u=N^{2-Q}$ and $Q = 2n+2$, we find that
$$
\frac{\nabla_0 N}{N} = - \frac1{2n} \, \frac{\nabla_0 u}{u} = -\frac1{2\bar{n}+1} \left( (P - \bar{n} \frac{B}{C} R) \vec{v} - (Q + \bar{n} \frac{t}{C} R) \vec{w} - \bar{n} R \vec{u} \right)
$$
and
$$
\frac{||\nabla_0 N||_h^2}{N^2} = \frac1{4n^2} \, \frac{||\nabla_0 u||_h^2}{u^2} = \frac1{(2\bar{n}+1)^2} \left(  \frac{2B-A}{2C(C+B)} + \bar{n} \frac{2(AB-AC+BC)}{C^2(C+A)} + \bar{n}^2\frac{2B}{C(C+A)} \right) \, .
$$

Before proceeding with the proof of Theorem \ref{th:anisotropic-Heis-groups}, we take a brief detour to confirm that $u$ is indeed $\cL$-harmonic in the complement of the origin. This result has already been proved in \cite{bdz:anisotropic}. From \eqref{eq:Laplacian-log-v} we see that
$$
\frac{\cL u}{u} = \diver_0 \left( \frac{\nabla_0 u}{u} \right) + \frac{||\nabla_0 u||_h^2}{u^2}.
$$
In order to compute the first term on the right hand side of this equation, we need to compute the values of $\nabla_0 P$, $\nabla_0 Q$, and $\nabla_0 R$, as well as $\diver_0(\vec{u})$, $\diver_0(\vec{v})$, and $\diver_0(\vec{w})$. We record these in the following lemma, whose proof is left as an exercise for the reader.

\begin{lemma}\label{lem:derivatives-of-PQR-uvw}
For the vector-valued functions $\vec{u}$, $\vec{v}$, and $\vec{w}$ introduced in Lemma \ref{lem:derivatives-ABCt}, we have
$$
\diver_0(\vec{u}) = 1, \quad \diver_0(\vec{v}) = 2n-1 = 2\bar{n}, \quad \mbox{and} \quad \diver_0(\vec{w}) = 0.
$$
For the functions $P$, $Q$, and $R$ introduced in \eqref{eq:PQR}, we have
$$
\nabla_0 P = \frac{4B^2-BC-2C^2}{2C^4}\vec{v} - \frac{2B-C}{2C^4} t \vec{w},
$$
$$
\nabla_0 Q = \frac{C-4B}{2C^4} t \vec{v} - \frac{C^3+3BC^2-3B^2C-4B^3}{2C^4(C+B)} \vec{w},
$$
and
$$
\nabla_0 R = -\frac{B}{C(C+A)^2} \vec{v} - \frac{t}{C(C+A)^2} \vec{w} - \frac1{(C+A)^2} \vec{u}.
$$
\end{lemma}

We now compute
\begin{equation*}\begin{split}
\diver_0 \left( \frac{\nabla_0 u}{u} \right) 
&= (P - \bar{n} \frac{B}{C} R) \diver_0(\vec{v}) + (Q + \bar{n} \frac{t}{C} R) \diver_0(\vec{w}) - \bar{n} R \diver_0(\vec{u}) \\
& \quad + \langle \nabla_0 P - \bar{n} \nabla_0(\frac{B}{C}R),\vec{v} \rangle + \langle \nabla_0 Q + \bar{n} \nabla_0(\frac{t}{C}R) ,\vec{w} \rangle - \bar{n} \langle \nabla_0 R,\vec{u} \rangle . 
\end{split}\end{equation*}
Using the formulas in the previous lemma and \eqref{eq:grad-log-u-2}, and after extensive algebraic manipulation, we conclude that $ \diver_0 ( \nabla_0 u/u ) = - ||\nabla_0 u||_h^2 / u^2$ and so $\cL u = 0$.

We now continue with the proof of Theorem \ref{th:anisotropic-Heis-groups}. From \ref{eq:infinity-Laplacian-log-v} we see that
\begin{equation}\begin{split}\label{eq:LinftyN}
\frac{\cL_\infty N}{N^3} 
= \cL_\infty(\log(N)) + \frac{||\nabla_0 N||_h^4}{N^4} 
= \left( -\frac1{8n^3} \right) \cL_\infty(\log(u)) + \left( \frac1{16n^4} \right)  \frac{||\nabla_0 u||_h^4}{u^4}.
\end{split}\end{equation}
Our next task is to compute
$$
\cL_\infty(\log(u)) = \frac12 \left\langle \nabla_0 \left( \frac{||\nabla_0 u||_h^2}{u^2} \right), \frac{\nabla_0 u}{u} \right\rangle
$$
for which we begin with the following proposition.

\begin{proposition}\label{prop:tech}
$\nabla_0(||\nabla_0 u||_h^2/u^2) = E_v(A,B,C) \vec{v} + E_w(A,B,C) t \vec{w} + E_u(A,B,C) \vec{u}$, where
\begin{equation*}\begin{split}
E_v(A,B,C) &= \frac{(AB-2BC+AC)+2t^2}{2C^3(C+B)} + \bar{n} \frac{A(C-B)(C+2B)(A+2C)+C^2(C^2-2B^2)}{C^4(C+A)^2} \\
&\qquad + \bar{n}^2 \frac{2(C+A)t^2-2B^2C}{C^3(C+A)^2} \, ,
\end{split}\end{equation*}
$$
E_w(A,B,C) = \frac{(A-2B)(B+2C)}{2C^3(C+B)^2} - \bar{n} \frac{2BC^2+A(C-2B)(2C-A)}{C^4(C+A)^2} - \bar{n}^2 \frac{2B(A+2C)}{C^3(C+A)^2} \, ,
$$
and
$$
E_u(A,B,C) = - \frac1{2C(C+B)} - \bar{n} \frac1{(C+A)^2} - \bar{n}^2 \frac{2B}{C(C+A)^2} \, .
$$
\end{proposition}

\noindent The proof is pure algebra using prior expressions for the gradients and divergences of relevant quantities, and is left to the reader. We have grouped the terms in this expression according to the vectors $\vec{u}$, $\vec{v}$, and $\vec{w}$, to facilitate the computation of the inner product of this expression with $\nabla_0 u/u$; see \eqref{eq:grad-log-u-1}. We now compute this inner product and use \eqref{eq:uvw-1}, \eqref{eq:uvw-2}, and \eqref{eq:uvw-3}. The result is a cubic polynomial in $\bar{n}$ which we write in brief as follows:
\begin{equation}\label{eq:Linftyu2}
\cL_\infty(\log(u)) = F_0(A,B,C) + F_1(A,B,C) \bar{n} + F_2(A,B,C) \bar{n}^2 + F_3(A,B,C) \bar{n}^3.
\end{equation}
We refer to \eqref{eq:grad-log-u-2} which we write in the form
$$
\frac{||\nabla_0 u||_h^2}{u^2} = G_0(A,B,C) + G_1(A,B,C) \bar{n} + G_2(A,B,C) \bar{n}^2.
$$
Returning to \eqref{eq:LinftyN}, we collect all of the above computations and write
\begin{equation}\label{eq:LinftyN2}
\frac{\cL_\infty N}{N^3} = \left( -\frac1{8n^3} \right) (F_0 + F_1 \bar{n} + F_2 \bar{n}^2 + F_3 \bar{n}^3) + \left( \frac1{16n^4} \right) (G_0 + G_1 \bar{n} + G_2 \bar{n}^2 )^2.
\end{equation}
Recalling that $\bar{n} = n-\tfrac12$ we recognize the preceding expression as the sum of two rational functions in $n$, the first of which is the quotient of two cubic polynomials in $n$ and the second of which is the quotient of two quartic polynomials in $n$. We expand this expression as a Laurent series in $n$. The constant term is 
$$
-\frac1{8} F_3 + \frac1{16} G_2^2.
$$
From \eqref{eq:grad-log-u-2} we recall that $G_2 = 2B/C(C+A)$ and after algebraic manipulation we conclude that
$$
F_3(A,B,C) = \frac{AB^2+2B^2C-AC^2}{C^3(C+A)^2}.
$$
Hence
$$
G_2^2 - 2F_3 = \frac{4B^2}{C^2(C+A)^2} - \frac{2(AB^2+2B^2C-AC^2)}{C^3(C+A)^2} = \frac{2A(C^2-B^2)}{C^3(C+A)^2} = \frac{2At^2}{C^3(C+A)^2}.
$$
We note that $G_2^2 - 2F_3$ vanishes precisely along the subspace $\Sigma = \{(z,t) \in \G : t=0\}$. It follows that \eqref{eq:LinftyN2}, when restricted to elements of $\Sigma$, can be written in the form
$$
\frac{\cL_\infty N}{N^3}(z,0) = H_{-1}(A,B,C) \frac1n + H_{-2}(A,B,C) \frac1{n^2} + H_{-3}(A,B,C) \frac1{n^3} + H_{-4}(A,B,C) \frac1{n^4} .
$$
The quantities $H_{-j}$, $j=1,2,3,4$, are polynomial expressions in the quantities $F_0$, $F_1$, $F_2$, $F_3$, $G_0$, $G_1$, $G_2$, and hence are rational functions in the variables $A,B,C$, or alternatively in the variables $z_1$ and $z'$. Moreover, the denominators of these rational functions are monomials in the terms $C+A$, $C+B$, and $C$. When restricted to the set $\Sigma \cap S_N = \{(z,0):N(z,0)=1\}$, the terms $C+A$, $C+B$, and $C$ are all strictly bounded away from zero. It follows that
$$
N \cL_\infty N = (N^4) \left( \frac{\cL_\infty N}{N^3} \right)
$$
is uniformly bounded from above on $\Sigma \cap S_N$ by a constant multiple of $\tfrac1n$. This completes the proof of the upper bound in the estimate \eqref{eq:NLN-asymptotics}. 

\smallskip

To prove the lower bound, we must show that the expression $N \, \cL_\infty N$ takes on a value comparable to $\tfrac1n$ at some point of $\Sigma \cap S_N$. In fact, we will show that this conclusion holds for all points of the form $(0,z',0) \in \Sigma \cap S_N$. To this end, we must evaluate the expressions $H_{-j}$, $j=1,2,3,4$, at points of the form $(0,z',0) \in \Sigma\cap S_N$. Observe that for points of this type, we have $B=C=\tfrac12$ and $A=0$. In this situation, equation \eqref{eq:grad-log-u-2} reduces to
\begin{equation*}\begin{split}
\frac{||\nabla_0 u||_h^2}{u^2} = G_0(0,\tfrac12,\tfrac12) + G_1(0,\tfrac12,\tfrac12) \bar{n} + G_2(0,\tfrac12,\tfrac12) \bar{n}^2 = 1 + 4 \bar{n} + 4 \bar{n}^2 \,.
\end{split}\end{equation*}
Similarly, the formula in Proposition \ref{prop:tech} simplifies to
\begin{equation*}\begin{split}
\nabla_0 \left( \frac{||\nabla_0 u||_h^2}{u^2} \right) &= E_v(0,\tfrac12,\tfrac12) \vec{v} + E_w(0,\tfrac12,\tfrac12) t \vec{w} + E_u(0,\tfrac12,\tfrac12) \vec{u} \\
&= \left( - 2 - 4 \bar{n} - 8 \bar{n}^2 \right) \vec{v} + \left( - 1 - 4 \bar{n} - 8 \bar{n}^2 \right) \vec{u}
\end{split}\end{equation*}
while \eqref{eq:grad-log-u-1} reduces to
$$
\frac{\nabla_0 u}{u} = \left( - 1 - 2 \bar{n} \right) \vec{v} - 2 \bar{n} \vec{u}.
$$
Then
\begin{equation*}\begin{split}
\cL_\infty(\log(u)) = F_0 + F_1 \bar{n} + F_2 \bar{n}^2 + F_3 \bar{n}^3 = 1 + 4 \bar{n} + 8 \bar{n}^2 + 8 \bar{n}^3 \,.
\end{split}\end{equation*}
Returning to \eqref{eq:LinftyN2} and using the fact that $\bar{n} = n-\tfrac12$, we find that
\begin{equation*}\begin{split}
\frac{\cL_\infty N}{N^3} = \frac{1}{2n} - \frac{1}{4n^2} \ge \frac1{4n} \, .
\end{split}\end{equation*}
This completes the proof of the lower bound in \eqref{eq:NLN-asymptotics}, and hence completes the proof of Theorem \ref{th:anisotropic-Heis-groups}. \qed

\bibliographystyle{acm}
\bibliography{biblio}

\begin{thebibliography}{10}

\bibitem{mathstackexchange}
Inequality of {F}robenius norm for skew matrices.
\newblock {\it Math StackExchange} discussion, available at {\tt
  https://math.stackexchange.com/questions/708287/inequality-of-frobenius-norm-for-skew-matrices}.

\bibitem{bt:polar}
{\sc Balogh, Z.~M., and Tyson, J.~T.}
\newblock Polar coordinates in {C}arnot groups.
\newblock {\em Math.\ Z. 241\/} (2002), 697--730.

\bibitem{br:generalized-H-type}
{\sc Barilari, D., and Rizzi, L.}
\newblock Sharp measure contraction property for generalized {H}-type {C}arnot
  groups.
\newblock {\em Comm.\ Contemp.\ Math. 20}, 6 (2018), 1750081 (24 pages).

\bibitem{bgg:step2}
{\sc Beals, R., Gaveau, B., and Greiner, P.}
\newblock The {G}reen function of model step two hypoelliptci operators and the
  analysis of certain tangential {C}aucy-{R}iemann complexes.
\newblock {\em Adv.\ Math. 121\/} (1996), 288--345.

\bibitem{blu}
{\sc Bonfiglioli, A., Lanconelli, E., and Uguzzoni, F.}
\newblock {\em Stratified {L}ie groups and potential theory for their
  sub-{L}aplacians}.
\newblock Springer Monographs in Mathematics. Springer, Berlin, 2007.

\bibitem{bdz:anisotropic}
{\sc Bou~Dagher, E., and Zegarli\'nski, B.}
\newblock Coercive inqualities in higher-dimensional anisotropic {H}eisenberg
  groups.
\newblock {\em Anal.\ Math.\ Phys. 12}, 3 (2022).

\bibitem{bru:superposition1}
{\sc Brustad, K.~K.}
\newblock Superposition in the {$p$}-{L}aplace equation.
\newblock {\em Nonlinear Anal. 158\/} (2017), 23--31.

\bibitem{bru:superposition2}
{\sc Brustad, K.~K.}
\newblock Superposition of {$p$}-superharmonic functions.
\newblock {\em Adv.\ Calc.\ Var. 13}, 2 (2020), 155--177.

\bibitem{cdg:carnot}
{\sc Capogna, L., Danielli, D., and Garofalo, N.}
\newblock Capacitary estimates and the local behavior of solutions of nonlinear
  subelliptic equations.
\newblock {\em Amer.\ J. Math. 118\/} (1996), 1153--1196.

\bibitem{ct:anisotropic1}
{\sc Chang, D.-C., and Tie, J.}
\newblock Estimates for powers of sub-{L}aplacian on the non-isotropic
  {H}eisenberg group.
\newblock {\em J. Geom.\ Anal. 10}, 4 (2000), 653--678.

\bibitem{ct:anisotropic2}
{\sc Chang, D.-C., and Tie, J.}
\newblock Some differential operators related to the non-isotropic {H}eisenberg
  sub-{L}aplacian.
\newblock {\em Math.\ Nachr. 221\/} (2001), 19--39.

\bibitem{cz:another}
{\sc Crandall, M.~G., and Zhang, J.}
\newblock Another way to say harmonic.
\newblock {\em Trans.\ Amer.\ Math.\ Soc. 355}, 1 (2003), 241--263.

\bibitem{fol:explicit}
{\sc Folland, G.~B.}
\newblock A fundamental solution for a subelliptic operator.
\newblock {\em Bull.\ Amer.\ Math.\ Soc. 79\/} (1973), 373--376.

\bibitem{fol:subelliptic}
{\sc Folland, G.~B.}
\newblock Subelliptic estimates and function spaces on nilpotent {L}ie groups.
\newblock {\em Ark.\ Mat. 13}, 2 (1975), 161--207.

\bibitem{fs:hardy}
{\sc Folland, G.~B., and Stein, E.~M.}
\newblock {\em Hardy spaces on homogeneous groups}.
\newblock Princeton University Press, Princeton, NJ, 1982.

\bibitem{gt:h-type}
{\sc Garofalo, N., and Tyson, J.~T.}
\newblock Riesz potentials and {$p$}-superharmonic functions in {L}ie groups of
  {H}eisenberg type.
\newblock {\em Bull.\ Lond.\ Math.\ Soc. 44}, 2 (2012), 353--366.

\bibitem{gv:yamabe}
{\sc Garofalo, N., and Vassilev, D.}
\newblock Regularity near the characteristic set in the nonlinear {D}irichlet
  problem and conformal geometry of sub-{L}aplacians on {C}arnot groups.
\newblock {\em Math.\ Ann. 318\/} (2000), 453--516.

\bibitem{hh:carnot}
{\sc Heinonen, J., and Holopainen, I.}
\newblock Quasiregular maps on {C}arnot groups.
\newblock {\em J.\ Geom.\ Anal. 7}, 1 (1997), 109--148.

\bibitem{kap:h-type}
{\sc Kaplan, A.}
\newblock Fundamental solutions for a class of hypoelliptic {P}{D}{E} generated
  by composition of quadratic forms.
\newblock {\em Trans.\ Amer.\ Math.\ Soc. 258}, 1 (1980), 147--153.

\bibitem{kr:rings}
{\sc Kor\'anyi, A., and Reimann, H.~M.}
\newblock Horizontal normal vectors and conformal capacity of spherical rings
  in the heisenberg group.
\newblock {\em Bull.\ Sci.\ Math.\ (2)}, 1 (1987), 3--21.

\bibitem{lm:remarkable}
{\sc Lindqvist, P., and Manfredi, J.~J.}
\newblock Note on a remarkable superposition for a nonlinear equation.
\newblock {\em Proc. Amer. Math. Soc. 136}, 1 (2008), 133--140.

\bibitem{Ste}
{\sc Stein, E.~M.}
\newblock {\em Harmonic Analysis: Real-variable Methods, Orthogonality and
  Oscillatory Integrals}, vol.~43 of {\em Princeton Mathematical Series}.
\newblock Princeton University Press, Princeton, New Jersey, 1993.

\bibitem{tys:polar2}
{\sc Tyson, J.~T.}
\newblock Polar coordinates in {C}arnot groups {I}{I}.
\newblock Preprint, 2022.

\end{thebibliography}
\end{document}